\documentclass[reqno]{amsart}
\usepackage{graphicx}
\usepackage[dvips]{epsfig}
\usepackage{bm}
\numberwithin{figure}{section}

\newtheorem{theorem}{Theorem}[section]
\newtheorem{lemma}{Lemma}[section]
\newtheorem{definition}[lemma]{Definition}
\newtheorem{problem}{{\bf{Problem}}}[section]

\numberwithin{equation}{section}

\allowdisplaybreaks
\begin{document}

\title[Transonic shocks in cylinders]{Transonic shocks for 3-D axisymmetric compressible inviscid flows in cylinders}

\author{Hyangdong Park}
\address{Center for Mathematical Analysis and Computation (CMAC), Yonsei University, 50 Yonsei-Ro, Seodaemun-Gu, Seoul 03722, Republic of Korea}
\email{hyangdong.park@yonsei.ac.kr}

\author{Hyeongyu Ryu} 
\address{Department of Mathematics, POSTECH, 77 Cheongam-Ro, Nam-Gu, Pohang, Gyeongbuk 37673, Republic of Korea}
\email{hgryu@postech.ac.kr}




\date{\today}

\begin{abstract}
We establish the existence of an axisymmetric weak solution to the steady Euler system with a transonic shock, nonzero vorticity, and nonzero swirl in a three-dimensional cylinder.
When prescribing the supersonic solution in the upstream region by axisymmetric functions with variable entropy and variable angular momentum density(=swirl), we construct such a solution by using a Helmholtz decomposition of the velocity field and the method of iteration.
An iteration scheme is developed using a delicate decomposition of the Rankine-Hugoniot conditions on the transonic shock via Helmholtz decomposition.
\end{abstract}
\maketitle


\tableofcontents

\section{Introduction}
The steady inviscid compressible flow of ideal polytropic gas in $\mathbb{R}^3$ is governed by the steady Euler system (cf. \cite{bae2019contact3D,courant1999supersonic}):
\begin{equation}\label{E-System}
\left\{
\begin{split}
&\mbox{div}(\rho{\bf u})=0,\\
&\mbox{div}(\rho{\bf u}\otimes{\bf u}+p\,{\mathbb I}_3)=0\quad(\mathbb{I}_3\mbox{ : $3\times 3$ identity matrix}),\\
&\mbox{div}(\rho{\bf u}B)=0.
\end{split}
\right.
\end{equation}
In the system above, $\rho=\rho({\bf x})$, ${\bf u}=(u_1,u_2,u_3)({\bf x})$, $p=p({\bf x})$, and $B=B({\bf x})$ denote the density, velocity,  pressure, and the {\emph{Bernoulli invariant}} of the flow, respectively, at ${\bf x}=(x_1,x_2,x_3)\in\mathbb{R}^3$.
For a constant $\gamma>1$ called the {\emph{adiabatic exponent}}, $B$ is defined by 
\begin{equation*}\label{Ber-inv}
B:=\frac{1}{2}|{\bf u}|^2+\frac{\gamma p}{(\gamma-1)\rho}.
\end{equation*}

Let $\Omega\subset \mathbb{R}^3$ be an open connected set, and let $\Gamma$ be a non-self-intersecting $C^1$-surface dividing $\Omega$ into two disjoint open subsets $\Omega^{\pm}$ such that $\Omega=\Omega^-\cup\Gamma\cup \Omega^+$.
\begin{definition}
We define ${\bf U}=({\bf u},\rho, p)\in [L^{\infty}_{\rm loc}(\Omega)\cap C^1_{\rm loc}(\Omega^{\pm})\cap C^0_{\rm loc}(\Omega^{\pm}\cup \Gamma)]^5$ to be a {\emph{weak solution}} to \eqref{E-System} in $\Omega$ if the following properties are satisfied: 
For any test function $\xi\in C_0^{\infty}(\Omega)$ and $j=1,2,3$,
\begin{equation}\label{weak}
\int_{\Omega}\rho {\bf u}\cdot\nabla\xi d{\bf x}=\int_{\Omega}(\rho u_j{\bf u}+p{\bf e}_j)\cdot\nabla\xi d{\bf x}=\int_{\Omega}\rho {\bf u}B\cdot\nabla\xi d{\bf x}=0,
\end{equation}
where ${\bf e}_j$ is the unit vector in the $x_j$-direction.
\end{definition}
By the integration by parts, one can easily check that ${\bf U}$ satisfies \eqref{weak} if and only if 
\begin{itemize}
\item[$(w_1)$] ${\bf U}$ is a classical solution to \eqref{E-System} in $\Omega^{\pm}$;
\item[$(w_2)$] ${\bf U}$ satisfies the Rankine-Hugoniot conditions
\begin{eqnarray}
\label{R-H-1}&&[\rho{\bf u}\cdot{\bf n}]_{\Gamma}=0,\quad[\rho({\bf u}\cdot{\bf n})^2+p]_{\Gamma}=0,\quad[\rho{\bf u}\cdot{\bf n}B]_{\Gamma}=0,\\
\label{R-H-3}&&\rho({\bf u}\cdot{\bf n})[{\bf u}\cdot{\bm \tau}_k]_{\Gamma}=0\quad\text{for all\,\, $k=1,2$,}
\end{eqnarray}
where $[\,\cdot\,]_{\Gamma}$ is defined by
$$[F({\bf x})]_{\Gamma}:=\left.F({\bf x})\right|_{\overline{\Omega^-}}-\left.F({\bf x})\right|_{\overline{\Omega^+}}\quad\mbox{for}\quad {\bf x}\in\Gamma,$$
${\bf n}$ is a unit normal vector field on $\Gamma$, 
and ${\bm \tau}_k$ ($k=1,2$) are unit tangent vector fields on $\Gamma$ such that they are linearly independent at each point on $\Gamma$.
\end{itemize}

Assume that $\rho>0$ in $\Omega$. Then, the condition in \eqref{R-H-3} is satisfied if either ${\bf u}\cdot{\bf n}=0$ on $\Gamma$, or $[{\bf u}\cdot{\bm \tau}_k]_{\Gamma}=0$ for all $k=1,2$.
We are interested in the latter case. 
For the former case, one can refer to \cite{bae2019contact,bae2019contact3D} and the references therein.

\begin{definition}
\label{definition-wsol}

We define ${\bf U}=( {\bf u}, \rho, p)\in [L^{\infty}_{\rm loc}(\Omega)\cap C^1_{\rm loc}(\Omega^{\pm})\cap C^0_{\rm loc}(\Omega^{\pm}\cup \Gamma)]^5$ to be a weak solution to \eqref{E-System} in $\Omega$ with a {\emph{transonic shock $\Gamma$}} if we have the following properties:
\begin{itemize}
\item[(i)] $\Gamma$ is a non-self-intersecting $C^1$-surface dividing $\Omega$ into two open subsets $\Omega^{\pm}$ such that $\Omega=\Omega^+\cup\Gamma\cup \Omega^-$;

\item[(ii)] ${\bf U}$ satisfies $(w_1)$, i.e., ${\bf U}$ is a classical solution to \eqref{E-System} in $\Omega^{\pm}$;

\item[(iii)] (Positivity of density) $\rho>0$ in $\overline{\Omega}$;

\item[(iv)] (Rankine-Hugoniot conditions) ${\bf U}$ satisfies \eqref{R-H-1} in $(w_2)$ and $[{\bf u}\cdot{\bm\tau}_k]_{\Gamma}=0$  for all $k=1,2$;

\item[(v)] (Transonic speed)
$|{\bf u}|>c$ (supersonic speed) in $\Omega^-$ and $|{\bf u}|<c$ (subsonic speed) in $\Omega^+$ for the sound speed $c:=\sqrt{\frac{\gamma p}{\rho}}$;

\item[(vi)] (Admissibility) ${\bf u}|_{\overline{\Omega^-}\cap \Gamma}\cdot{\bf n}>{\bf u}|_{\overline{\Omega^+}\cap \Gamma}\cdot{\bf n}>0$ for the unit normal vector field ${\bf n}$ on $\Gamma$ pointing toward $\Omega^+$.
\end{itemize}
\end{definition}

In this paper, we study the existence of a weak solution to \eqref{E-System} with a transonic shock in the sense of Definition \ref{definition-wsol} in a three-dimensional cylinder (Figure \ref{shock-def}).

\begin{figure}[!h]
\centering
\includegraphics[scale=0.73]{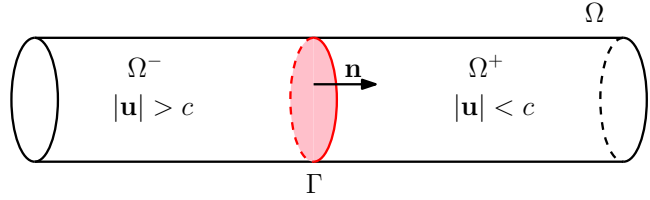}
\caption{Transonic shock}\label{shock-def}
\end{figure}


Let $(x,r,\theta)$ denote the cylindrical coordinates of ${\bf x}=(x_1,x_2,x_3)\in\mathbb{R}^3$, i.e., 
\begin{equation*}
(x_1,x_2,x_3)=(x,r\cos\theta,r\sin\theta),\quad r\ge0,\quad\theta\in\mathbb{T},
\end{equation*}
where $\mathbb{T}$ is a one-dimensional torus with period $2\pi$. 
Then, any function $f({\bf x})$ and vector-valued function ${\bf F}({\bf x})$ can be represented by
\begin{equation*}
f({\bf x})=f(x,r,\theta)\quad\mbox{and}\quad{\bf F}({\bf x})=F_x(x,r,\theta){\bf e}_x+F_r(x,r,\theta){\bf e}_r+F_{\theta}(x,r,\theta){\bf e}_\theta
\end{equation*}
for orthonormal vectors
\begin{equation*}
{\bf e}_x=(1,0,0),\quad {\bf e}_r=(0,\cos\theta,\sin\theta),\quad{\bf e}_{\theta}=(0,-\sin\theta, \cos\theta).
\end{equation*}
\begin{definition} {\cite{bae2019contact3D,bae20183}}
\begin{itemize}
\item[(i)] A function $f({\bf x})$ is {\emph{axisymmetric}} if its value is independent of $\theta$.
\item[(ii)] A vector-valued function ${\bf F}({\bf x})$ is {\emph{axisymmetric}} if each of functions $F_x$, $F_r$, and $F_\theta$ is axisymmetric.
\end{itemize}
\end{definition}
The purpose of this paper is to prove the existence of an axisymmetric weak solution to the steady Euler system \eqref{E-System} with a transonic shock in a three-dimensional cylinder.
The existence, uniqueness, and stability of transonic shocks for 3-D steady flows in cylindrical nozzles were studied in \cite{chen2003multidimensional, chen2004steady, chen2007existence,chen2009uniqueness,chen2008trans,chen2008transonic}. 
In \cite{chen2003multidimensional,chen2004steady,chen2007existence}, the existence and stability of multidimensional transonic shocks for potential flows were established.
In \cite{chen2009uniqueness}, authors proved the uniqueness of solutions with a transonic shock in a class of transonic shock solutions, which are not necessarily small perturbations of the background solution, for potential flow. 
In \cite{chen2008trans, chen2008transonic}, the existence and stability of the perturbed compressible flow including a transonic shock were studied for the prescribed pressure at the exit up to a constant. 
For studies on multidimensional transonic shocks in diverging nozzles, see  \cite{bae2011transonic,li2010transonic,liu2016stability,liu2009global,xin2008transonic} and the references cited therein.

In this paper, we establish the existence of an axisymmetric weak solution to the steady Euler system with a transonic shock, nonzero vorticity, and nonzero swirl in a three-dimensional cylinder.
When prescribing the supersonic solution in the upstream region by axisymmetric functions with variable entropy and variable angular momentum density(=swirl), we construct such a solution by using a Helmholtz decomposition of the velocity field and the method of iteration.
An iteration scheme is developed using a delicate decomposition of the Rankine-Hugoniot conditions on the transonic shock via Helmholtz decomposition.
This approach, using Helmholtz decomposition, is a new attempt to investigate multidimensional transonic shock solutions to the Euler system.

The structure of the paper is as follows: 
In Section \ref{3D-sec-Main}, the main problem and theorem are stated as Problem \ref{Shock-Prob} and Theorem \ref{MainThm}, respectively.
In Section \ref{3D-sec-Hel}, we reformulate the main problem via Helmholtz decomposition, and state its solvability as Theorem \ref{Thm-HD}.
In Section \ref{sec-proof-HD}, we prove Theorem \ref{Thm-HD} by using the method of iteration.
Finally, in Section \ref{sec-proof-main}, we prove Theorem \ref{MainThm} by Theorem \ref{Thm-HD}.

\section{Problems and Main Results}\label{3D-sec-Main}
Define a cylinder $\mathcal{N}$ by 
\begin{equation*}\label{def-N-cyl}
\mathcal{N}:=\left\{(x_1,x_2,x_3)\in\mathbb{R}^3:\mbox{ }-1<x_1<1,\mbox{ }0\le \sqrt{x_2^2+x_3^2}<1\right\}.
\end{equation*}
As defined in the previous section, let $(x,r,\theta)$ be the cylindrical coordinates of $(x_1,x_2,x_3)\in\mathbb{R}^3$, i.e.,
$$(x_1,x_2,x_3)=(x,r\cos\theta,r\sin\theta),\quad r\ge 0,\quad\theta\in\mathbb{T},$$
where $\mathbb{T}$ denotes a one-dimensional torus with period $2\pi$.
Then, the entrance $\Gamma_{\rm en}$, exit $\Gamma_{\rm ex}$, and the wall $\Gamma_{\rm w}$ of $\mathcal{N}$ are defined as 
\begin{equation*}
\Gamma_{\rm en}:=\partial\mathcal{N}\cap\{x=-1\},\,\,
\Gamma_{\rm ex}:=\partial\mathcal{N}\cap\{x=1\},\,\,\Gamma_{\rm w}:=\partial\mathcal{N}\cap\{r=1\}.
\end{equation*}

We first construct a simple solution.
Let $(u_0^-,\rho_0^-,p_0^-)$ be positive constants satisfying 
\begin{equation*}
u_0^->\sqrt{\frac{\gamma p_0^-}{\rho_0^-}}.
\end{equation*}
Define a function $U_0$ by 
\begin{equation*}
U_0(x_1,x_2,x_3):=\left\{
\begin{split}
({\bf u}_0^-,\rho_0^-,p_0^-)\quad&\mbox{for}\quad x_1<0,\\
({\bf u}_0^+,\rho_0^+,p_0^+)\quad&\mbox{for}\quad x_1>0,
\end{split}\right.
\end{equation*}
where ${\bf u}_0^{\pm}:=(u_0^{\pm},0,0)$ and $(u_0^+,\rho_0^+,p_0^+)$ are positive constants defined by 
\begin{equation*}
\left.\begin{split}
u_0^+&:=\frac{2(\gamma-1)}{(\gamma+1)u_0^-}\left(\frac{1}{2}|u_0^-|^2+\frac{\gamma p_0^-}{(\gamma-1)\rho_0^-}\right),\\
\rho_0^+&:=\frac{\rho_0^-u_0^-}{u_0^+},\quad p_0^+:=\rho_0^-|u_0^-|^2+p_0^--\rho_0^+|u_0^+|^2.
\end{split}\right.
\end{equation*}
Then one can easily check that $U_0$ is a weak solution to \eqref{E-System} in $\mathcal{N}$ with a transonic shock $\mathfrak{S}_0:={\mathcal{N}}\cap\{x=0\}$ (Figure \ref{back}). 
For later use, we set 
\begin{equation}\label{B0}
\begin{split}
&S_0^{\pm}:=\frac{p_0^{\pm}}{(\rho_0^{\pm})^{\gamma}},\quad B_0:=\frac{1}{2}|u_0^-|^2+\frac{\gamma p_0^-}{(\gamma-1)\rho_0^-},\\
&\varphi_0^{\pm}({\bf x}):=u_0^{\pm}x_1\quad\mbox{for}\,\,{\bf x}=(x_1,x_2,x_3)\in\mathcal{N}.
\end{split}
\end{equation}

\begin{figure}[!h]
\centering
\includegraphics[scale=0.68]{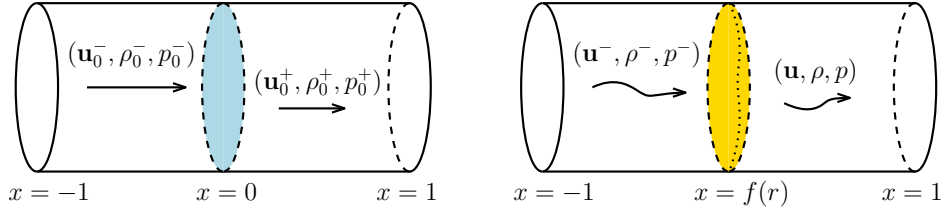}
\caption{Left: Background solution, Right: Problem \ref{Shock-Prob}}\label{back}
\end{figure}

Before we state our problem and main results, we introduce some H\"older norms defined as follows:

(i) ({\emph{Standard H\"older norms}})
 Let $\Omega\subset\mathbb{R}^n$ be an open bounded connected set.
For $\alpha\in(0,1)$ and $m\in\mathbb{Z}^+$, define 
the standard H\"older norms by 
\begin{equation*}
\begin{split}
&\|u\|_{m,\Omega}:=\sum_{0\le|\beta|\le m}\sup_{{\bf x}\in\Omega}|D^{\beta}u({\bf x})|,\quad [u]_{m,\alpha,\Omega}:=\sum_{|\beta|=m}\sup_{\substack{{\bf x,y}\in\Omega,\\{\bf x}\ne{\bf y}}}\frac{|D^{\beta}u({\bf x})-D^{\beta}u({\bf y})|}{|{\bf x}-{\bf y}|^{\alpha}},\\
&\|u\|_{m,\alpha,\Omega}:=\|u\|_{m,\Omega}+[u]_{m,\alpha,\Omega}.
\end{split}
\end{equation*}
Here, $D^{\beta}$ denotes $\partial_{x_1}^{\beta_1}\dots\partial_{x_n}^{\beta_n}$ for a multi-index $\beta=(\beta_1,\ldots,\beta_n)$ with $\beta_j\in\mathbb{Z}^+$ and $|\beta|=\sum_{j=1}^n\beta_j$. 

(ii) ({\emph{Weighted H\"older norms}}) Let $\Gamma$ be a closed portion of $\partial\Omega$. For ${\bf x}, {\bf y}\in\Omega$ and $k\in\mathbb{R}$, set
\begin{equation*}
\delta_{\bf x}:=\inf_{{\bf z}\in\Gamma}|{\bf x}-{\bf z}|\quad\mbox{and}\quad\delta_{{\bf x,y}}:=\min(\delta_{\bf x},\delta_{\bf y}),
\end{equation*}
and define the weighted H\"older norms by 
\begin{eqnarray*}
\begin{split}
&\|u\|_{m,0,\Omega}^{(k,\Gamma)}:=\sum_{0\le|\beta|\le m}\sup_{{\bf x}\in\Omega}\delta_{\bf x}^{\max(|\beta|+k,0)}|D^{\beta}u({\bf x})|,\\
&[u]_{m,\alpha,\Omega}^{(k,\Gamma)}:=\sup_{|\beta|=m}\sup_{\substack{{\bf x,y}\in\Omega,\\{\bf x}\ne{\bf y}}}\delta_{\bf x,y}^{\max(m+\alpha+k,0)}\frac{|D^{\beta}u({\bf x})-D^{\beta}u({\bf y})|}{|{\bf x}-{\bf y}|^{\alpha}},\\
&\|u\|_{m,\alpha,\Omega}^{(k,\Gamma)}:=\|u\|_{m,0,\Omega}^{(k,\Gamma)}+[u]_{m,\alpha,\Omega}^{(k,\Gamma)}.
\end{split}
\end{eqnarray*}
$C^{m,\alpha}_{(k,\Gamma)}(\Omega)$ denotes the completion of the set of all smooth functions whose $\|\cdot\|_{m,\alpha,\Omega}^{(k,\Gamma)}$ norms are finite. 

Our goal is to solve the following problem.
\begin{problem}\label{Shock-Prob}
Let $({\bf u}^-,\rho^-,p^-)$ be an axisymmetric supersonic solution of \eqref{E-System} in $\mathcal{N}$ with $B=B_0$ for a constant $B_0>0$ defined in \eqref{B0}, and suppose that 
\begin{equation*}\label{super-ass}
{\bf u}^-\cdot{\bf e}_r=0\quad\mbox{on}\quad\Gamma_{\rm w}. 
\end{equation*}
Find a weak solution $U=({\bf u},\rho, p)$ to  \eqref{E-System} with a transonic shock
$$\mathfrak{S}_f:\,x=f(r)\quad\mbox{(Figure \ref{back})}$$
in the sense of Definition \ref{definition-wsol} in $\mathcal{N}$ such that we have the following properties:
\begin{itemize}
\item[(a)] (Positivity of density and velocity along $x$-direction)
$$\rho>0\,\quad\mbox{and}\,\quad{\bf u}\cdot{\bf e}_x>\frac{u_0^+}{2}>0\,\quad\mbox{in}\,\quad\overline{\mathcal{N}}.$$
\item[(b)] In $\mathcal{N}_f^-:=\mathcal{N}\cap\{x<f(r)\}$,
$U=({\bf u}^-,\rho^-,p^-)$ holds.
\item[(c)] In $\mathcal{N}_f^+:=\mathcal{N}\cap\{x>f(r)\}$,
 ${|{\bf u}|}<c\,\,\mbox{for the sound speed}\,\, c=\sqrt{\frac{\gamma p}{\rho}}.$
 \item[(d)] 
$U$ satisfies the boundary condition
$${\bf u}\cdot{\bf e}_r=0\quad\mbox{on}\quad\Gamma_{\rm ex}\cup\Gamma_{\rm w}.$$
\item[(e)] 
$U$ satisfies the Rankine-Hugoniot conditions
\begin{equation*}\label{RH-re}
\begin{split}
&[{\bf u}\cdot{\bm\tau}_f]_{\mathfrak{S}_f}=[\rho {\bf u}\cdot{\bf n}_f]_{\mathfrak{S}_f}=[B]_{\mathfrak{S}_f}=[\rho({\bf u}\cdot{\bf n}_f)^2+p]_{\mathfrak{S}_f}=0\quad\mbox{on}\quad\mathfrak{S}_f,
\end{split}
\end{equation*}
where ${\bm \tau}_f$ and ${\bf n}_f$ denote a unit tangent vector field and unit normal vector field on $\mathfrak{S}_f$, respectively.
\end{itemize}
\end{problem}

 Problem \ref{Shock-Prob} can be rewritten as the following free boundary problem:
\begin{problem}\label{re-Prob}
Under the same assumptions of Problem \ref{Shock-Prob}, find a radial function
$$f:B_1(0)\left(:=\left\{{\bf y}\in\mathbb{R}^2:\,|{\bf y}|<1\right\}\right)\longrightarrow(-\frac{1}{4},\frac{1}{4})$$ and a weak solution $U=({\bf u},\rho,p)$ to \eqref{E-System} in $\mathcal{N}_f^+:=\mathcal{N}\cap\{x>f(r)\}$ such that the following properties hold:
\begin{itemize}
\item[(a)] The properties (a) and (c) in Problem \ref{Shock-Prob} hold in $\overline{\mathcal{N}_f^+}$.
\item[(b)] On $\Gamma_{{\rm w},f}^+(:=\partial\mathcal{N}_f^+\cap\Gamma_{\rm w})\cup\Gamma_{\rm ex}$, $U$ satisfies the boundary conditions
\begin{equation*}
{\bf u}\cdot{\bf e}_r=0.
\end{equation*}
\item[(c)] On $\mathfrak{S}_f:\,x=f(r)$, $U$ satisfies the boundary conditions
\begin{equation}\label{RH-re}
\left\{\begin{split}
&{\bf u}\cdot{\bm \tau}_f={\bf u}^-\cdot{\bm\tau}_f,\\
&\rho{\bf u}\cdot{\bf n}_f=\rho^-{\bf u}^-\cdot{\bf n}_f,\\
&\rho({\bf u}\cdot{\bf n}_f)^2+p=\rho^-({\bf u}^-\cdot{\bf n}_f)^2+p^-,
\end{split}\right.
\end{equation}
where ${\bm \tau}_f$ and ${\bf n}_f$ denote a unit tangent vector field and unit normal vector field on $\mathfrak{S}_f$, respectively. 
\item[(d)] The Bernoulli function $B$ is a constant function, 
$$B\equiv B_0\quad\mbox{in}\quad\overline{\mathcal{N}_f^+},$$
where $B_0$ is given in \eqref{B0}.
\end{itemize}
\end{problem}

The main theorem of this paper is as follows:

\begin{theorem}\label{MainThm}
Let $({\bf u}^-,\rho^-,p^-)$ be from Problem \ref{Shock-Prob}.
For any fixed $\alpha\in(\frac{1}{6},1)$, there exists a small constant $\sigma_1>0$ depending only on $(u_0^{\pm},\rho_0^{\pm},p_0^{\pm},\gamma,\alpha)$ so that if
\begin{equation}\label{3D-rem11}
\sigma({\bf u}^-,\rho^-,p^-):=\|({\bf u}^-,\rho^-,p^-)-({\bf u}_0^-,\rho_0^-,p_0^-)\|_{1,\alpha,\mathcal{N}}\le \sigma_1,
\end{equation}
 then there exists an axisymmetric solution $U=({\bf u},\rho,p)$ of Problem \ref{re-Prob} with a transonic shock $x=f(r)$ satisfying
 \begin{equation}\label{Thm2.1-uniq-est}
\|f\|_{2,\alpha,B_1(0)}^{(-1-\alpha,\partial B_1(0))}+\|({\bf u},\rho,p)-({\bf u}_0^+,\rho_0^+,p_0^+)\|_{1,\alpha,\mathcal{N}_f^+}^{(-\alpha,\Gamma_{{\rm w},f}^+)}\le C\sigma,
\end{equation}
where the constant $C>0$ depends only on $(u_0^{\pm},\rho_0^{\pm},p_0^{\pm},\gamma,\alpha)$.
\end{theorem}

\section{Reformulation of Problem \ref{re-Prob} via Helmholtz decomposition}\label{3D-sec-Hel}
Suppose that the smooth solution $({\bf u},\rho, p)$ of \eqref{E-System} is axisymmetric with $B=B_0$  for a constant $B_0>0$ defined in \eqref{B0}.
Then 
\begin{equation*}
{\bf u}=u_x(x,r){\bf e}_x+u_r(x,r){\bf e}_r+u_{\theta}(x,r){\bf e}_{\theta},\quad\rho=\rho(x,r),\quad p=p(x,r),
\end{equation*}
and the system \eqref{E-System} can be rewritten as follows:
\begin{equation}\label{3D-ang}
\left\{\begin{split}
&\partial_x(\rho u_x)+\partial_r(\rho u_r)+\frac{\rho u_r}{r}=0,\\
&\rho(u_x\partial_x+u_r\partial_r)u_r-\frac{\rho \Lambda^2}{r^3}+\partial_r (S\rho^{\gamma})=0,\\
&\rho(u_x\partial_x+u_r\partial_r)S=0,\\
&\rho(u_x\partial_x+u_r\partial_r)\Lambda=0.
\end{split}
\right.
\end{equation}
Here, $S:=\frac{p}{\rho^{\gamma}}$ denotes the entropy  and 
\begin{equation*}
\label{definition-Lambda}
\Lambda(x,r):=ru_{\theta}(x,r)
\end{equation*}
denotes the angular momentum density.

For a radial function $f:B_1(0)\longrightarrow(-\frac{1}{4},\frac{1}{4})$ to be determined, 
we express the velocity field ${\bf u}$ as 
\begin{equation*}
{\bf u}=\nabla\varphi+\mbox{curl}{\bf V}\quad\mbox{in}\quad\mathcal{N}_f^+
\end{equation*}
for axisymmetric functions
\begin{equation*}
\varphi({\bf x})=\varphi(x,r),\quad {\bf V}({\bf x})=h(x,r){\bf e}_r+\psi(x,r){\bf e}_{\theta}.
\end{equation*}
Suppose that $(\varphi, {\bf V})$ are $C^2$ in $\mathcal{N}_f^+$.
Then a straightforward computation yields  
\begin{equation}\label{3D-u}
\begin{split}
{\bf u}&=\left(\partial_x\varphi+\frac{1}{r}\partial_r(r\psi)\right){\bf e}_x+(\partial_r\varphi-\partial_x\psi){\bf e}_r+\left(\frac{\Lambda}{r}\right){\bf e}_{\theta}\\
&=:{\bf q}(r,\psi,D\psi,D\varphi,\Lambda)\quad\mbox{for}\quad D=(\partial_x,\partial_r).
\end{split}
\end{equation}
To simplify notations, we set
\begin{equation}\label{def-T}
{\bf t}(r,\psi,D\psi,\Lambda):={\bf q}(r,\psi,D\psi,D\varphi,\Lambda)-\nabla\varphi\left(=\mbox{curl}{\bf V}\right).
\end{equation}
Then, as in \cite[Section 3]{bae2019contact3D}, we can rewrite  \eqref{3D-ang} as a system for $(\varphi,\psi,S,\Lambda)$:
\begin{equation}\label{3D-H}
\left\{\begin{split}
&\mbox{div}\left(H(S,{\bf q}){\bf q}\right)=0,\\
&-\Delta(\psi{\bf e}_{\theta})=G(S,\Lambda,\partial_r S,\partial_r\Lambda,{\bf t},\nabla\varphi){\bf e}_{\theta},\\
&H(S,{\bf q}){\bf q}\cdot\nabla S=0,\\
& H(S,{\bf q}){\bf q}\cdot\nabla \Lambda=0,\\
\end{split}
\right.
\end{equation}
with
\begin{equation*}
{\bf q}={\bf q}(r,\psi,D\psi,D\varphi,\Lambda),\quad\text{and}\quad {\bf t}={\bf t}(r,\psi,D\psi,\Lambda)
\end{equation*}
for $(H, G)$ defined by
\begin{equation}\label{def-H-G}
\left.\begin{split}
&H(\eta,{\bf q}):=\left[\frac{\gamma-1}{\gamma\eta}\left(B_0-\frac{1}{2}|{\bf q}|^2\right)\right]^{1/(\gamma-1)},\\
&G(\eta_1,\eta_2,\eta_3,\eta_4,{\bf t},{\bf v}):=\frac{1}{({\bf t}+{\bf v})\cdot{\bf e}_x}\left(\frac{H^{\gamma-1}(\eta_1, {\bf t}+{\bf v})}{\gamma-1}\eta_3+\frac{\eta_2}{r^2}\eta_4\right)
\end{split}
\right.
\end{equation}
for $\eta\in\mathbb{R}$, ${\bf q}\in\mathbb{R}^3$, $\eta_1,\eta_2,\eta_3,\eta_4\in\mathbb{R}$, and ${\bf t},{\bf v}\in\mathbb{R}^3$.

For an axisymmetric supersonic solution $({\bf u}^-,\rho^-,p^-)$ given in Problem \ref{Shock-Prob}, we find $(S^-,\Lambda^-,\varphi^-,\psi^-)$ such that 
\begin{equation*}
{\bf u}^-={\bf q}(r,\psi^-,D\psi^-,D\varphi^-,\Lambda^-),\quad \rho^-=H(S^-,{\bf u}^-),\quad p^-=S^-(\rho^-)^{\gamma}\quad\mbox{in}\quad\mathcal{N}.
\end{equation*}
For that purpose, we first solve the following linear boundary value problem:
\begin{equation}\label{psi-super}
-\Delta{\bf W}^-=\left(\partial_x ({\bf u}^-\cdot{\bf e}_r)-\partial_r ({\bf u}^-\cdot{\bf e}_x)\right){\bf e}_{\theta}\quad\mbox{in}\quad\mathcal{N}
\end{equation}
with boundary conditions
\begin{equation}\label{BC-super}
\partial_x{\bf W}^-={\bf 0}\quad\mbox{on}\quad \Gamma_{\rm en}\cup\Gamma_{\rm ex}, \quad {\bf W}^-={\bf 0}\quad\mbox{on}\quad\Gamma_{\rm w}\cup\{r=0\}.
\end{equation}
By  \cite[Proposition 3.3]{bae20183}, 
the unique axisymmetric solution ${\bf W}^-\in C^{2,\alpha}(\overline{\mathcal{N}})$ of the boundary value problem \eqref{psi-super}-\eqref{BC-super} has the form of  
\begin{equation*}\label{psi--}
{\bf W}^-=\psi^-(x,r){\bf e}_{\theta}
\end{equation*}
and the axisymmetric function $\psi^-$ is $C^2$ as a function of $(x,r)$ in a two-dimensional rectangle $(-1,1)\times(0,1)$.
With such $\psi^-$, we define axisymmetric functions $(\varphi^-,S^-,\Lambda^-)$ by 
\begin{equation}\label{super-HD}
\begin{split}
&\varphi^-(x,r):=\int_0^x \left({\bf u}^-\cdot{\bf e}_x-\frac{1}{r}\partial_r(r\psi^-)\right)(y,r)dy,\\
&S^-(x,r):=\frac{p^-}{(\rho^-)^{\gamma}}(x,r),\quad\Lambda^-(x,r):=r{\bf u}^-\cdot{\bf e}_{\theta}(x,r)\quad\mbox{in}\quad\mathcal{N}.
\end{split}
\end{equation}
Then it follows from \eqref{psi-super} and \eqref{super-HD} that 
\begin{equation*}
\partial_r\varphi^-={\bf u}^-\cdot{\bf e}_r+\partial_x\psi^-\quad\mbox{in}\quad\mathcal{N}.
\end{equation*}
By direct computations, one can check that 
 there exists a constant $C_{\ast}>0$ depending only on $(u_0^-,\rho_0^-,p_0^-,\gamma,\alpha)$ such that 
\begin{equation}\label{const}
\sigma(S^-,\Lambda^-,\varphi^-,\psi^- )\le C_{\ast}\sigma({\bf u}^-,\rho^-,p^-),
\end{equation}
where $\sigma({\bf u}^-,\rho^-,p^-)$ is defined in \eqref{3D-rem11} and 
\begin{equation*}\label{pertur}
\sigma(S^-,\Lambda^-,\varphi^-,\psi^- ):=\|(S^-,\Lambda^-)-(S_0^-,0)\|_{1,\alpha,\mathcal{N}}+\|\varphi^--\varphi_0^-\|_{2,\alpha,\mathcal{N}}+\|\psi^-{\bf e}_{\theta}\|_{2,\alpha,\mathcal{N}}
\end{equation*}
for
$(S_0^-,\varphi_0^-)$ defined by \eqref{B0}.

Now we derive boundary conditions for $(f,S,\Lambda,\varphi,\psi)$ to satisfy the conditions (b)-(c) in Problem \ref{re-Prob}.

{{(i) Boundary conditions on $\Gamma_{{\rm w},f}^+\cup\Gamma_{\rm ex}$: }} 

We prescribe boundary conditions for $(\varphi,\psi)$ on $\Gamma_{{\rm w},f}^+\cup\Gamma_{\rm ex}$ as
\begin{equation}\label{w-ex-bd}
\left\{\begin{split}
\partial_r\varphi=0\quad\mbox{and}\quad \psi=0\quad&\mbox{on}\quad\Gamma_{{\rm w},f}^+,\\
\varphi=\varphi_0^+\quad\mbox{and}\quad \partial_x\psi=0\quad&\mbox{on}\quad\Gamma_{\rm ex}
\end{split}\right.
\end{equation}
for $\varphi_0^+$ given by \eqref{B0}
so that the boundary condition (b) in Problem \ref{re-Prob} holds on $\Gamma_{{\rm w},f}^+\cup\Gamma_{\rm ex}$.

\smallskip

{{(ii) The Rankine-Hugoniot conditions on $\mathfrak{S}_f$: }}  

In terms of $(f,S,\Lambda, \varphi,\psi)$,  the Rankine-Hugoniot conditions in \eqref{RH-re} become 
\begin{eqnarray}
&\label{HD-RH1}&{\bf q}\cdot{\bm\tau}_{f}={\bf u}^-\cdot{\bm\tau}_f,\quad{\bf q}\cdot{\bf e}_{\theta}={\bf u}^-\cdot{\bf e}_{\theta},\\
&\label{HD-RH2}&H(S,{\bf q})({\bf q}\cdot{\bf n}_f)=\rho^-({\bf u}^-\cdot{\bf n}_f),\\
&\label{HD-RH3}&H(S,{\bf q})({\bf q}\cdot{\bf n}_f)^2+SH^{\gamma}(S,{\bf q})=\rho^-({\bf u}^-\cdot{\bf n}_f)^2+p^-
\end{eqnarray}
for
\begin{equation*}
{\bm \tau}_{f}=\frac{f'(r){\bf e}_x+{\bf e}_r}{\sqrt{1+|f'(r)|^2}},\quad
{\bf n}_{f}=\frac{{\bf e}_x-f'(r){\bf e}_r}{\sqrt{1+|f'(r)|^2}}.
\end{equation*}
We derive the conditions of $(f,S,\Lambda,\varphi,\psi)$ to satisfy \eqref{HD-RH1}-\eqref{HD-RH3}.

On $\mathfrak{S}_f$, if $(f,\Lambda,\varphi,\psi)$ satisfy
\begin{equation}\label{v-sh}
\left.\begin{split}
&\varphi=\varphi^-,\quad\Lambda=\Lambda^-,\\
&-\nabla(\psi{\bf e}_{\theta})\cdot{\bf n}_f=\left(\frac{\left(-\frac{1}{r}\psi+\frac{1}{r}\psi^-+\partial_r\psi^-\right)f'(r)-\partial_x\psi^-}{\sqrt{1+|f'(r)|^2}}\right){\bf e}_{\theta},
\end{split}\right.
\end{equation}
then we have 
\begin{equation*}
\nabla\varphi\cdot{\bm\tau}_f=\nabla\varphi^-\cdot{\bm\tau}_f,\quad{\bf t}\cdot{\bm\tau}_f=\frac{1}{r}\partial_r(r\psi^-)f'(r)-\partial_x\psi^-,\quad{\bf q}\cdot{\bf e}_{\theta}=\frac{\Lambda^-}{r},
\end{equation*}
from which we get the conditions in \eqref{HD-RH1}.
Note that the first condition in \eqref{v-sh} is equivalent to 
\begin{equation*}\label{free-cond}
f(r)=\frac{(\varphi-\varphi_0^+)-(\varphi^--\varphi_0^-)}{u_0^--u_0^+}(f(r),r)\quad\mbox{for}\quad 0\le r\le 1.
\end{equation*}
This will be used to find the location of the transonic shock $\mathfrak{S}_f$ (See Lemma \ref{shock-lemma}).


On $\mathfrak{S}_f$, if $(f,S,\Lambda,\varphi,\psi)$ satisfy 
\begin{eqnarray}
&\nonumber&\nabla\varphi\cdot{\bf n}_f=\frac{K_s(f')}{{\bf u}^-\cdot{\bf n}_f}-{\bf t}\cdot{\bf n}_f,\\
&\label{S-shock}&S=\left(\rho^-({\bf u}^-\cdot{\bf n}_f)^2+p^--\rho^-K_s(f')\right)\left(\frac{\rho^-({\bf u}^-\cdot{\bf n}_f)^2}{K_s(f')}\right)^{-\gamma}
\end{eqnarray}
for 
\begin{equation}\label{def-KS}
K_s(f'):=\frac{2(\gamma-1)}{\gamma+1}\left(\frac{1}{2}|{\bf u^-}\cdot{\bf n}_f|^2+\frac{\gamma p^-}{(\gamma-1)\rho^-}\right),
\end{equation}
then one can directly check that 
\begin{eqnarray}
&\label{KS}&{\bf q}\cdot{\bf n}_f=\frac{K_s(f')}{{\bf u}^-\cdot{\bf n}_f},\\
&\label{S-RH}&\rho^-K_s(f')+S\left(\frac{\rho^-({\bf u}^-\cdot{\bf n}_f)^2}{K_s(f')}\right)^{\gamma}=\rho^-({\bf u}^-\cdot{\bf n}_f)^2+p^-.
\end{eqnarray}
By \eqref{S-shock}-\eqref{KS} and the definition of $H$ in \eqref{def-H-G}, we have
 \begin{equation}\label{rho-shock}
 H(S,{\bf q})=\frac{\rho^-({\bf u}^-\cdot{\bf n}_f)^2}{K_s(f')}\quad\mbox{on}\quad\mathfrak{S}_f.
 \end{equation}
Then it follows from \eqref{KS} and \eqref{rho-shock} that the condition in \eqref{HD-RH2} holds.
Finally, the condition in \eqref{HD-RH3} holds by \eqref{HD-RH2} and \eqref{KS}-\eqref{rho-shock}.

We gather all the boundary conditions for $(f,S,\Lambda,\varphi,\psi)$ on $\mathfrak{S}_f$ as follows: 
\begin{equation}\label{boundary-HD}
\left\{
\begin{split}
	\begin{split}
	&S=S_{\rm sh}(f'),\quad \Lambda=\Lambda^-,\quad\varphi=\varphi^-,\\
	&-\nabla\varphi\cdot{\bf n}_f=-\frac{K_s(f')}{{\bf u}^-\cdot{\bf n}_f}+{\bf t}\cdot{\bf n}_f,\\
	&-\nabla(\psi{\bf e}_{\theta})\cdot{\bf n}_f=\mathcal{A}(r,\psi,f'){\bf e}_{\theta}\qquad\mbox{on}\quad\mathfrak{S}_f,
	\end{split}
\end{split}\right.
\end{equation} 
with 
\begin{equation}\label{def-A}
\begin{split}
&S_{\rm sh}(f'):=\left(\rho^-({\bf u}^-\cdot{\bf n}_f)^2+p^--\rho^-K_s(f')\right)\left(\frac{\rho^-({\bf u}^-\cdot{\bf n}_f)^2}{K_s(f')}\right)^{-\gamma},\\
&\mathcal{A}(r,\psi,f'):=\frac{\left(-\frac{1}{r}\psi+\frac{1}{r}\psi^-+\partial_r\psi^-\right)f'(r)-\partial_x\psi^-}{\sqrt{1+|f'(r)|^2}}.
\end{split}
\end{equation}


\begin{theorem}\label{Thm-HD}
Let $({\bf u}^-,\rho^-,p^-)$ be from Problem \ref{Shock-Prob}.
For simplicity of notations,  let $\sigma$ denote  $\sigma({\bf u}^-,\rho^-,p^- )$ defined in \eqref{3D-rem11}.
Then, for any $\alpha\in(\frac{1}{6},1)$, there exists a small constant $\sigma_2>0$ depending only on 
\sloppy$(u_0^{\pm},\rho_0^{\pm},p_0^\pm,\gamma,\alpha)$ so that if
\begin{equation*}\label{enu-sigma}
\sigma\le\sigma_2,
\end{equation*}
then the free boundary problem \eqref{3D-H} with boundary conditions \eqref{w-ex-bd} and \eqref{boundary-HD} has an axisymmetric solution $(f, S, \Lambda, \varphi, \psi)$ that satisfies
\begin{equation}\label{Thm-HD-est}
\begin{split}
\|f\|_{2,\alpha,B_1(0)}^{(-1-\alpha,\partial B_1(0))}
&+\|(S,\Lambda)-(S_0^+,0)\|_{1,\alpha,\mathcal{N}_f^+}^{(-\alpha,\Gamma_{{\rm w},f}^+)}\\
&+\|\varphi-\varphi_0^+\|_{2,\alpha,\mathcal{N}_f^+}^{(-1-\alpha,\Gamma_{{\rm w},f}^+)}+\|\psi{\bf e}_{\theta}\|_{2,\alpha,\mathcal{N}_f^+}^{(-1-\alpha,\Gamma_{{\rm w},f}^+)}\le C\sigma,
\end{split}
\end{equation}
where $(S_0^+,\varphi_0^+)$ are given in \eqref{B0} and the constant $C>0$ depends only on $(u_0^{\pm},\rho_0^{\pm},p_0^\pm,\gamma,\alpha)$.
\end{theorem}

Hereafter, a constant $C$ is said to be chosen depending only on the data if $C$ is chosen depending only on $(u_0^{\pm},\rho_0^{\pm},p_0^{\pm},\gamma,\alpha)$.
Unless otherwise specified, each estimate constant $C$ is regarded to be depending only on the data for the rest of the paper. 

\section{Proof of Theorem \ref{Thm-HD}}\label{sec-proof-HD}

This section is devoted to proving Theorem \ref{Thm-HD} by using the method of iteration.
In Section \ref{sec-sets}, we first define iterations sets and prove one lemma which concerns the location of the transonic shock.
In Sections \ref{sec-lem1}-\ref{sec-lem2}, we prove two lemmas that will be used to prove Theorem \ref{Thm-HD}.
Finally, in Section \ref{sec-Thm-hd}, we prove Theorem \ref{Thm-HD}.

\subsection{Iteration sets}\label{sec-sets}
For a fixed $\alpha\in(\frac{1}{6},1)$, we define  iteration sets as follows:\\
(i) Iteration set for $\varphi-\varphi_0^+(=:\phi)$: For $M_1>0$ to be determined later, we define
\begin{equation}\label{Ite-set-phi}
\mathcal{I}(M_1):=\left\{\phi=\phi(x,r)\in C_{(-1-\alpha,\Gamma_{\rm w})}^{2,\alpha}({\mathcal{N}_{-1/2}^+}):\begin{split}
	&\|\phi\|_{2,\alpha,\mathcal{N}_{-1/2}^+}^{(-1-\alpha,\Gamma_{\rm w})}\le M_1\sigma
	\end{split}\right\},
\end{equation}
where $\mathcal{N}_{-1/2}^+:=\mathcal{N}\cap\{-\frac{1}{2}<x<1\}$. \newline
(ii) Iteration set for $(S,\Lambda)$: For $M_2>0$ to be determined later, we define
\begin{equation}\label{Ite-set-T}
\mathcal{I}(M_2):=\mathcal{I}_1(M_2)\times\mathcal{I}_2(M_2)
\end{equation}
with
\begin{equation*}
\begin{split}
	&\mathcal{I}_1(M_2):=\left\{S=S(x,r)\in C_{(-\alpha,\Gamma_{\rm w})}^{1,\alpha}({\mathcal{N}_{-1/2}^+}) :\, \|S-S_0^+\|^{(-\alpha,\Gamma_{\rm w})}_{1,\alpha,\mathcal{N}_{-1/2}^+}\le M_2\sigma\right\},\\
	&\mathcal{I}_2(M_2):=\left\{\Lambda=r\mathcal{V}(x,r)\in C_{(-\alpha,\Gamma_{\rm w})}^{1,\alpha}({\mathcal{N}_{-1/2}^+}):
		\begin{split}
			&\|\mathcal{V}\|^{(-\alpha,\{r=1\})}_{1,\alpha,\Omega_{-1/2}^+}\le M_2\sigma,\\
			&\mathcal{V}(x,0)=0, \forall x\in[-\frac{1}{2},1]
		\end{split}\right\},
\end{split}
\end{equation*}
where $\Omega_{-1/2}^+$ is a two-dimensional rectangular domain defined by
\begin{equation*}
\Omega_{-1/2}^+:=\left\{(x,r)\in\mathbb{R}^2:\, -\frac{1}{2}<x<1,\, 0<r<1\right\}.
\end{equation*}
(iii) Iteration set for $\psi$: For $M_3>0$ to be determined later, we define
\begin{equation}\label{Ite-set-psi}
\mathcal{I}(M_3):=\left\{\psi=\psi(x,r)\in C_{(-1-\alpha,\{r=1\})}^{2,\alpha}({\Omega_{-1/2}^+}):\begin{split}
	&\|\psi\|^{(-1-\alpha,\{r=1\})}_{2,\alpha,\Omega_{-1/2}^+}\le M_3\sigma,\\
	&\psi(x,1)=0,\,\partial_r^k\psi(x,0)=0\\
	&\mbox{for }k=0,2,\,\forall x\in[-\frac{1}{2},1]\\
	\end{split}
	\right\}.
\end{equation}

\begin{lemma}\label{shock-lemma}
For a constant $C_{\ast}>0$ in \eqref{const}, suppose that 
\begin{equation}\label{sigma-free}
\sigma\le\frac{u_0^--u_0^+}{8(M_1+C_{\ast})}(=:\sigma_3).
\end{equation}
Then, for each $\phi\in\mathcal{I}(M_1)$, there exists a radial function $f:B_1(0)\longrightarrow(-\frac{1}{4},\frac{1}{4})$ such that 
\begin{equation*}
f(r)=\frac{\phi-(\varphi^--\varphi_0^-)}{u_0^--u_0^+}(f(r),r).
\end{equation*}
Moreover, $f$ satisfies the estimate
\begin{equation*}
\|f\|_{2,\alpha,B_1(0)}^{(-1-\alpha,\partial B_1(0))}\le C(1+M_1)\sigma
\end{equation*}
for a constant $C>0$ depending only on the data.
\end{lemma}

\begin{proof}
For $\sigma\le \sigma_3$, one can directly check that 
\begin{eqnarray}
&\nonumber&\partial_x(\phi+\varphi_0^+-\varphi^-)=\partial_x(\phi+\varphi_0^--\varphi^-)-(u_0^--u_0^-)\le -\frac{u_0^--u_0^+}{2}<0,\quad\\
&\nonumber&(\phi+\varphi_0^+-\varphi^-)(-\frac{1}{4},r)\ge\frac{u_0^--u_0^+}{8(M_1+C_{\ast})}>0,\quad\forall r\in[0,1],\\
&\nonumber&(\phi+\varphi_0^+-\varphi^-)(\frac{1}{4},r)\le-\frac{u_0^--u_0^+}{8(M_1+C_{\ast})}<0,\quad\forall r\in[0,1].
\end{eqnarray}
Then the implicit function theorem implies that there exists a radial function $f:B_1(0)\longrightarrow(-\frac{1}{4},\frac{1}{4})$ satisfying
\begin{equation}\label{impli}
(\phi+\varphi_0^+-\varphi^-)(f(r),r)=0
\end{equation}
and 
\begin{equation*}
\|f\|_{2,\alpha,B_1(0)}^{(-1-\alpha,\partial B_1(0))}\le C(1+M_1)\sigma.
\end{equation*}
By the definitions of $\varphi_0^\pm$ given in \eqref{B0}, \eqref{impli} is equivalent to 
\begin{equation*}\label{f-diff}
f(r)=\frac{\phi-(\varphi^--\varphi_0^-)}{u_0^--u_0^+}(f(r),r).
\end{equation*}
The proof is completed.
\end{proof}

\subsection{Free boundary problem for $\varphi$}\label{sec-lem1}
In this subsection, we solve the following free boundary problem.
\begin{problem}\label{free-Pro}
For each  $(S_{\ast},\Lambda_{\ast},\psi_{\ast})\in\mathcal{I}(M_2)\times\mathcal{I}(M_3)$,  
set 
\begin{equation*}
{\bf q}_{\ast}:={\bf q}(r,\psi_{\ast},D\psi_{\ast},D\varphi,\Lambda_{\ast}),\quad{\bf t}_{\ast}:={\bf t}(r,\psi_{\ast},D\psi_{\ast},\Lambda_{\ast})
\end{equation*}
for $({\bf q},{\bf t})$ given by \eqref{3D-u} and \eqref{def-T}, respectively. Find $(f,\varphi)$ satisfying 
\begin{equation}\label{varphi-eq}
\left\{\begin{split}
\mbox{div\,}(H(S_{\ast},{\bf q}_{\ast}){\bf q}_{\ast})=0\quad&\mbox{in}\quad\mathcal{N}_{f}^+,\\
\varphi=\varphi^-,\quad-\nabla\varphi\cdot{\bf n}_{f}=-\frac{K_s(f')}{{\bf u}^-\cdot{\bf n}_f}+{\bf t}_{\ast}\cdot{\bf n}_f\quad&\mbox{on}\quad\mathfrak{S}_{f},\\
\partial_r\varphi=0\quad&\mbox{on}\quad\Gamma_{{\rm w},f}^+,\\
\varphi=\varphi_0^+\quad&\mbox{on}\quad\Gamma_{\rm ex},
\end{split}\right.
\end{equation}
where $H$ and $K_s(f')$ are given in \eqref{def-H-G} and \eqref{def-KS}, respectively.
\end{problem}

\begin{lemma}\label{lemma-free}
There exists a small constant $\sigma_4>0$ depending only on the data and $(M_2, M_3)$ so that if 
$$\sigma\le\sigma_4,$$
then, for each  $(S_{\ast},\Lambda_{\ast},\psi_{\ast})\in\mathcal{I}(M_2)\times\mathcal{I}(M_3)$,  Problem \ref{free-Pro}  has a unique axisymmetric solution $(f,\varphi)$ that satisfies
\begin{equation}\label{lem2-est}
\|f\|_{2,\alpha,B_1(0)}^{(-1-\alpha,\partial B_1(0))}+\|\varphi-\varphi_0^+\|^{(-1-\alpha,\Gamma_{{\rm w},f}^+)}_{2,\alpha,\mathcal{N}_f^+}\le C(1+M_2+M_3)\sigma,
\end{equation}
where the constant $C>0$ depends only on the data.
\end{lemma}

\begin{proof} The proof of Lemma \ref{lemma-free} is divided into five steps.

{\bf 1.}
Fix $\phi_{\ast}\in\mathcal{I}(M_1)$. 
By Lemma \ref{shock-lemma}, if $\sigma\le\sigma_3$, then there exists a radial function $f:B_1(0)\longrightarrow(-\frac{1}{4},\frac{1}{4})$ such that
\begin{equation*}
f(r)=\frac{\phi_{\ast}-(\varphi^--\varphi_0^-)}{u_0^--u_0^+}(f(r),r)\quad\mbox{and}\quad
\|f\|_{2,\alpha,B_1(0)}^{(-1-\alpha,\partial B_1(0))}\le C(1+M_1)\sigma.
\end{equation*}

{\bf 2.} (Linearized boundary value problem for $\varphi$)
For a function $H$ defined in \eqref{def-H-G}, 
we define $\widetilde{H}$ and ${\bf A}=(A_1,A_2,A_3)$ as follows:
\begin{equation*}
\widetilde{H}(\xi, {\bf s},{\bf v}):=H(\xi,{\bf s}+{\bf v}),\quad
A_j(\xi,{\bf s},{\bf v}):=\widetilde{H}(\xi, {\bf s},{\bf v})s_j\quad (j=1,2,3),
\end{equation*}
for $\xi\in\mathbb{R}$, ${\bf s}=(s_1,s_2,s_3)\in\mathbb{R}^3$, and ${\bf v}=(v_1,v_2,v_3)\in\mathbb{R}^3$.
Let us set 
\begin{equation}\label{aij-def}
{\bf V}_0:=(S_0^+,D\varphi_0^+,{\bf 0})\quad\mbox{and}\quad a_{ij}:=\partial_{s_j}A_i({\bf V}_0)\quad\mbox{for}\quad i,j=1,2,3.
\end{equation}
Then the constant matrix $[a_{ij}]_{i,j=1}^3$ is diagonal and there exists a positive constant $\nu\in(0,1/10]$ satisfying
\begin{equation*}\label{3D-nu-aij}
0<\nu< a_{ii}<\frac{1}{\nu}\quad\mbox{for all}\quad i=1,2,3.
\end{equation*}

We set $\phi:=\varphi-\varphi_0^+$ and rewrite the equation
$\mbox{div}\left(H(S,{\bf q}){\bf q}\right)=0$ as
\begin{equation*}
\mathcal{L}(\phi)=\mbox{div}{\bf F}(S-S_0^+,D\phi,{\bf t}),
\end{equation*}
where $\mathcal{L}$ and ${\bf F}=(F_1,F_2,F_3)$ are defined by
\begin{equation*}\label{def-F}
\left.\begin{split}
\mathcal{L}(\phi):=&\sum_{i=1}^3\partial_i(a_{ii}\partial_i\phi),\\
F_i(Q):=&-\widetilde{H}({\bf V}_0+Q)v_i-\int_0^1 D_{\xi,\bf v}A_i({\bf V}_0+tQ)dt\cdot(\xi,{\bf v})\\
&-{\bf s}\cdot\int_0^1D_{\bf s}A_i({\bf V}_0+tQ)-D_{\bf s}A_i({\bf V}_0)dt,
\end{split}
\right.
\end{equation*}
with  $Q=(\xi, {\bf s},{\bf v})\in\mathbb{R}\times(\mathbb{R}^3)^2$.
In the above, $\partial_{x_i}$ is abbreviated as $\partial_i$.
By the boundary conditions for $\varphi$ in \eqref{varphi-eq} and definition of $\varphi_0^+$ given in \eqref{B0}, the boundary conditions for $\phi$ on $\partial\mathcal{N}_f^+$ become
\begin{equation*}\label{boundary}
\left\{\begin{split}
(a_{ii}\partial_i\phi)\cdot(-{\bf n}_f)=\mathfrak{B}({\bf t}_{\ast},f',\nabla\phi)\quad&\mbox{on}\quad\mathfrak{S}_f,\\
(a_{ii}\partial_i\phi)\cdot{\bf e}_r=0\quad&\mbox{on}\quad\Gamma_{{\rm w},f}^+,\\
\phi=0\quad&\mbox{on}\quad\Gamma_{\rm ex}
\end{split}\right.
\end{equation*}
with
\begin{equation}\label{def-B}
\mathfrak{B}({\bf t},f',\nabla\phi):=\,a_{11}\left(-\frac{K_s(f')}{{\bf u}^-\cdot{\bf n}_f}+{\bf t}\cdot{\bf n}_f+\nabla\varphi_0^+\cdot{\bf n}_f\right)+\sum_{j=2}^3(a_{11}-a_{jj})\partial_j\phi({\bf n}_f\cdot{\bf e}_j).
\end{equation}

We consider the linear boundary value problem 
\begin{equation}\label{phi-equ}
\left\{\begin{split}
\mathcal{L}(\phi)=\mbox{div\,}{\bf F}_{\ast}\quad&\mbox{in}\quad\mathcal{N}_{f}^+,\\
(a_{ii}\partial_i\phi)\cdot(-{\bf n}_f)=\mathfrak{B}_{\ast}\quad&\mbox{on}\quad\mathfrak{S}_f,\\
(a_{ii}\partial_i\phi)\cdot{\bf e}_r=0\quad&\mbox{on}\quad\Gamma_{{\rm w},f}^+,\\
\phi=0\quad&\mbox{on}\quad\Gamma_{\rm ex},
\end{split}\right.
\end{equation}
with
\begin{equation*}
{\bf F}_{\ast}:={\bf F}(S_{\ast}-S_0^+,D\phi_{\ast},{\bf t}_{\ast})\quad\mbox{and}\quad\mathfrak{B}_{\ast}:=\mathfrak{B}({\bf t}_{\ast},f',\nabla\phi_{\ast}).
\end{equation*}
The standard elliptic theory (cf. Evans \cite{evans2002partial}, Gilbarg-Trudinger \cite{gilbarg2015elliptic}) yields that \eqref{phi-equ} has a unique weak solution $\phi\in H^1(\mathcal{N}_f^+)$ satisfying 
\begin{equation*}
\mathfrak{L}[\phi,\zeta]=\langle({\bf F}_{\ast},\mathfrak{B}_{\ast}),\zeta\rangle\quad\mbox{ for all $\zeta\in\left\{\zeta\in H^1(\mathcal{N}_f^+):\,\zeta=0\mbox{ on }\Gamma_{\rm ex}\right\}$}
\end{equation*}
where 
\begin{equation*}
\begin{split}
&\mathfrak{L}[\phi,\zeta]:=\int_{\mathcal{N}_f^+} \sum_{i=1}^{3} a_{ii}(\partial_i\phi)(\partial_i\zeta) d{\bf x},\\
&\langle({\bf F}_{\ast},\mathfrak{B}_{\ast}),\zeta\rangle:=\int_{\mathcal{N}_f^+}{\bf F}_{\ast}\cdot \nabla\zeta d{\bf x}-\int_{\partial\mathcal{N}_f^+\backslash\Gamma_{\rm ex}}({\bf F}_{\ast}\cdot{\bf n}_{\rm out})\zeta \,dS+\int_{\mathfrak{S}_f}\mathfrak{B}_{\ast}\zeta\,dS.
\end{split}
\end{equation*}
Furthermore, $\phi$ satisfies 
\begin{equation*}
\|\phi\|_{H^1(\mathcal{N}_f^+)}\le C\left(\|{\bf F}_{\ast}\|_{0,\alpha,\mathcal{N}_{f}^+}+\|\mathfrak{B}_{\ast}\|_{0,\alpha,\mathfrak{S}_f}\right).\end{equation*}

{\bf 3.} (Estimate of $\phi$)
For each ${\bf x}_0\in\overline{\mathcal{N}_f^+}$ and $\eta\in\mathbb{R}$ with $0<\eta<\frac{1}{10}$, let us set 
\begin{equation*}
\begin{split}
&B_\eta({\bf x}_0):=\{{\bf x}\in \mathbb{R}^3:|{\bf x}_0-{\bf x}|<\eta\},\quad D_\eta({\bf x}_0):=B_{\eta}({\bf x}_0)\cap\mathcal{N}_f^+,\\
&(\nabla\phi)_{{\bf x}_0,\eta}:=\frac{1}{|D_{\eta}({\bf x}_0)|}\int_{D_{\eta}({\bf x}_0)}\nabla\phi\, d{\bf x}.
\end{split}
\end{equation*} 
Since there exists a constant $\lambda_0\in(0,1/10)$ such that 
\begin{equation*}
\lambda_0\le\frac{|D_\eta({\bf x}_0)|}{|B_\eta({\bf x}_0)|}\le\frac{1}{\lambda_0},
\end{equation*}
we will follow the proofs of \cite[Theorem 3.13]{han2011elliptic} and \cite[Lemma 3.5]{bae2016subsonic} to get 
\begin{equation}\label{1alpha-est}
\int_{D_\eta({\bf x})}|\nabla\phi-(\nabla\phi)_{{\bf x},\eta}|^2 d{\bf x}\le C\left(\|{\bf F}_{\ast}\|^{(-\alpha,\Gamma_{{\rm w},f}^+)}_{1,\alpha,\mathcal{N}_{f}^+}+\|\mathfrak{B}_{\ast}\|^{(-\alpha,\partial\mathfrak{S}_f)}_{1,\alpha,\mathfrak{S}_f}\right)^2 {\eta}^{3+2\alpha}
\end{equation}
for any ${\bf x}\in\overline{\mathcal{N}_f^+}$. 
Once \eqref{1alpha-est} is proved, we obtain from \cite[Theorem 3.1]{han2011elliptic} that 
\begin{equation}\label{aalpha}
\|\phi\|_{1,\alpha,\mathcal{N}_{f}^+}\le C\left(\|{\bf F}_{\ast}\|^{(-\alpha,\Gamma_{{\rm w},f}^+)}_{1,\alpha,\mathcal{N}_{f}^+}+\|\mathfrak{B}_{\ast}\|^{(-\alpha,\partial\mathfrak{S}_f)}_{1,\alpha,\mathfrak{S}_f}\right).
\end{equation}
By the scaling argument and Schauder estimate with \eqref{aalpha}, we get
\begin{equation}\label{phi-est-lin}
\|\phi\|^{(-1-\alpha,\Gamma_{{\rm w},f}^+)}_{2,\alpha,\mathcal{N}_{f}^+}\le C\left(\|{\bf F}_{\ast}\|^{(-\alpha,\Gamma_{{\rm w},f}^+)}_{1,\alpha,\mathcal{N}_{f}^+}+\|\mathfrak{B}_{\ast}\|^{(-\alpha,\partial\mathfrak{S}_f)}_{1,\alpha,\mathfrak{S}_f}\right).
\end{equation}

We prove \eqref{1alpha-est} only for the case ${\bf x}\in\Gamma_{{\rm w},f}^+\cap\mathfrak{S}_f$. 
Fix ${\bf x}_0\in\Gamma_{{\rm w},f}^+\cap\mathfrak{S}_f$ and $\chi\in\mathbb{R}$ with $0<\chi<\frac{1}{10}$.
Let $\phi_{\rm h}$ be a weak solution  of the following problem
\begin{equation*}
\left\{\begin{split}
\mathcal{L}(\phi_{\rm h})=0\quad&\mbox{in}\quad D_{\chi}({\bf x}_0),\\
(a_{ii}\partial_i\phi_{\rm h})\cdot(-{\bf n}_f)=\mathfrak{B}_{\ast}({\bf x}_0)\quad&\mbox{on}\quad\partial D_{\chi}({\bf x}_0)\cap\mathfrak{S}_f,\\
(a_{ii}\partial_i\phi_{\rm h})\cdot{\bf e}_r=0\quad&\mbox{on}\quad\partial D_{\chi}({\bf x}_0)\cap\Gamma_{{\rm w},f}^+,\\
\phi_{\rm h}=\phi\quad&\mbox{on}\quad\partial D_{\chi}({\bf x}_0)\cap\mathcal{N}_f^+.
\end{split}\right.
\end{equation*}
Then $\phi_{\rm nh}:=\phi-\phi_{\rm h}$ satisfies 
\begin{equation*}
\int_{D_{\chi}({\bf x}_0)}a_{ij}\partial_i\phi_{\rm nh}\partial_j\xi \,d{\bf x}
=\int_{D_{\chi}({\bf x}_0)}(\mbox{div}{\bf F}_{\ast}) \zeta\, d{\bf x}+\int_{\partial D_{\chi}({\bf x}_0)\cap\mathfrak{S}_f}(\mathfrak{B}_{\ast}-\mathfrak{B}_{\ast}({\bf x}_0))\zeta\,dS
\end{equation*}
for any $\xi\in\left\{\zeta\in\mathcal{H}^1(D_{\chi}({\bf x}_0)):\, \zeta=0\mbox{ on }\partial D_{\chi}({\bf x}_0)\cap\mathcal{N}_f^+\right\}$. By taking the test function $\xi=\phi_{\rm nh}$ and using the  H\"older inequality, Sobolev inequality, and the Poincar\'e inequality, we get
\begin{equation}\label{Hol}
\begin{split}
\int_{D_{\chi}({\bf x}_0)}&|\nabla\phi_{\rm nh}|^2d{\bf x}\\
&\le C\left(\int_{D_{\chi}({\bf x}_0)}|\mbox{div}{\bf F}_{\ast}|^{6/5}d{\bf x}\right)^{5/3}+C\left(\chi\int_{\partial D_{\chi}({\bf x}_0)\cap\mathfrak{S}_f}|\mathfrak{B}_{\ast}-\mathfrak{B}_{\ast}({\bf x}_0)|^{2}\,dS\right)\\
&=:(i)+(ii).
\end{split}
\end{equation}
Since $\mbox{div}{\bf F}_{\ast}\in C^{0,\alpha}_{(1-\alpha,\Gamma_{{\rm w},f}^+)}(\mathcal{N}_f^+)$, 
the definition of the weighted H\"older norms yields that 
\begin{equation*}\label{weight-F}
|\mbox{div}{\bf F}_{\ast}({\bf x})|\le\delta_{\bf x}^{-1+\alpha}\|\mbox{div}{\bf F}_{\ast}\|_{0,\alpha,\mathcal{N}_f^+}^{(1-\alpha,\Gamma_{{\rm w},f}^+)}\quad\mbox{for}\quad{\bf x}\in\mathcal{N}_f^+,\quad\delta_{\bf x}=\inf_{{\bf z}\in\Gamma_{{\rm w},f}^+}|{\bf x}-{\bf z}|,
\end{equation*}
from which we obtain that 
\begin{equation}\label{(i)}
(i)\le C\chi^{3+2\alpha}\left(\|\mbox{div}{\bf F}_{\ast}\|_{0,\alpha,\mathcal{N}_f^+}^{(1-\alpha,\Gamma_{{\rm w},f}^+)}\right)^{2}\quad\mbox{for}\quad \alpha\in(\frac{1}{6},1).
\end{equation}
Since $\mathfrak{B}_{\ast}\in C^{1,\alpha}_{(-\alpha,\partial\mathfrak{S}_f)}(\mathfrak{S}_f)$, 
the definition of the weighted H\"older norms yields that 
\begin{equation}\label{est(ii)}
(ii)\le C\chi^{3+2\alpha}\left(\|\mathfrak{B}_{\ast}\|^{(-\alpha,\partial\mathfrak{S}_f)}_{1,\alpha,\mathfrak{S}_f}\right)^2.
\end{equation}
By \eqref{Hol}-\eqref{est(ii)}, we get
\begin{equation*}
\int_{D_{\chi}({\bf x}_0)}|\nabla\phi_{\rm nh}|^2d{\bf x}
\le C\chi^{3+2\alpha}\left(\|{\bf F}_{\ast}\|_{1,\alpha,\mathcal{N}_f^+}^{(-\alpha,\Gamma_{{\rm w},f}^+)}+\|\mathfrak{B}_{\ast}\|^{(-\alpha,\partial\mathfrak{S}_f)}_{1,\alpha,\mathfrak{S}_f}\right)^2.
\end{equation*}
Then \cite[Corollary 3.11]{han2011elliptic} implies for $0<\eta<\chi$
\begin{equation*}
\begin{split}
\int_{D_{\eta}({\bf x}_0)}|\nabla\phi-(\nabla\phi)_{{\bf x}_0,\eta}|^2d{\bf x}
&\le C\left(\frac{\eta}{\chi}\right)^5\int_{D_{\eta}({\bf x}_0)}|\nabla\phi-(\nabla\phi)_{{\bf x}_0,\eta}|^2d{\bf x}\\
&\quad+C\chi^{3+2\alpha}\left(\|{\bf F}_{\ast}\|_{1,\alpha,\mathcal{N}_f^+}^{(-\alpha,\Gamma_{{\rm w},f}^+)}+\|\mathfrak{B}_{\ast}\|^{(-\alpha,\partial\mathfrak{S}_f)}_{1,\alpha,\mathfrak{S}_f}\right)^2.
\end{split}
\end{equation*}
According to \cite[Lemma 3.4]{han2011elliptic}, we can replace $\chi^{3+2\alpha}$ by $\eta^{3+2\alpha}$. 
Hence the proof of \eqref{1alpha-est} is completed and we finally get the weighted $C^{2,\alpha}$-estimate \eqref{phi-est-lin} of $\phi$.

By direct computations, one can check that there exists a constant $\epsilon_1\in(0,\frac{1}{4})$ depending only on the data so that if 
\begin{equation*}
M_1\sigma+M_2\sigma+M_3\sigma\le\epsilon_1,
\end{equation*}
then we have 
\begin{equation}\label{est-FB}
\begin{split}
&\|{\bf F}_{\ast}\|_{1,\alpha,\mathcal{N}_{f}^+}^{(-\alpha,\Gamma_{{\rm w},f}^+)}\le C\left((M_1\sigma)^2+M_2\sigma+M_3\sigma\right),\\
&\|\mathfrak{B}_{\ast}\|_{1,\alpha,\mathfrak{S}_f}^{(-\alpha,\partial\mathfrak{S}_f)}\le C\left(\sigma+(M_1\sigma)^2+M_2\sigma+M_3\sigma\right).
\end{split}
\end{equation}
It follows from \eqref{phi-est-lin} and \eqref{est-FB} that
\begin{equation}\label{phi-est2}
\|\phi\|_{2,\alpha,\mathcal{N}_{f}^+}^{(-1-\alpha,\Gamma_{{\rm w},f}^+)}
\le C\left(\sigma+(M_1\sigma)^2+M_2\sigma+M_3\sigma\right).
\end{equation}

For any $\theta\in[0,2\pi)$, define a function $\phi^{\theta}$ by 
\begin{equation*}
\phi^{\theta}(x_1,x_2,x_3):=\phi(x_1,x_2\cos\theta-x_3\sin\theta,x_2\sin\theta+x_3\cos\theta).
\end{equation*}
Since $a_{22}=a_{33}$ by definitions of $a_{ij}$ in \eqref{aij-def}, $\phi^{\theta}$ is a solution to \eqref{phi-equ}.
The uniqueness of a solution to \eqref{phi-equ} implies that $\phi=\phi^{\theta}$. Hence $\phi$ is axisymmetric.

{\bf 4.} (Extension of $\phi$ into $\mathcal{N}_{-1/2}^+$)
To set an iteration mapping, we need to extend $\phi$ into $C_{(-1-\alpha,\Gamma_{\rm w})}^{2,\alpha}(\mathcal{N}_{-1/2}^+)$.
Define an extension function $\mathcal{E}_f(\phi)$ of $\phi$ by
\begin{equation}\label{extension}
\mathcal{E}_{f}(\phi):=E_f(\phi)\circ\mathfrak{T}_{f},
\end{equation}
where $\mathfrak{T}_{f}:\mathcal{N}_{2f-1}^+\longrightarrow\mathcal{N}$ is an invertible transformation defined by
\begin{equation*}
\mathfrak{T}_{f}(x_1,x_2,x_3)=\left(\frac{1}{1-f(r)}(x_1-1)+1,x_2,x_3\right)
\end{equation*}
and $E_f(\phi):\mathcal{N}\longrightarrow \mathbb{R}$ is a function defined by 
\begin{equation*}
E_f(\phi)(y_1,y_2,y_3):=\left\{\begin{split}
	\phi\circ\mathfrak{T}_{f}^{-1}(y_1,y_2,y_3)\quad&\mbox{for}\quad 0\le y_1<1,\\
	\sum_{i=1}^{3}c_i\left(\phi\circ\mathfrak{T}_{f}^{-1}\right)\left(-\frac{y_1}{i},y_2,y_3\right)\quad&\mbox{for}\quad-1<y_1<0,
	\end{split}\right.
\end{equation*}
where $c_1=6$, $c_2=-32$, and $c_3=27$, which are constants determined by the system of equations 
\begin{equation*}
\sum_{i=1}^3 c_i\left(-\frac{1}{i}\right)^m=1,\quad m=0,1,2.
\end{equation*}
Since $|f|<\frac{1}{4}$, the transformation $\mathfrak{T}_f$ is well-defined and $\mathcal{N}_{-1/2}^+\subset\mathcal{N}_{2f-1}^+$.
By a direct computation, we have 
\begin{equation}\label{est-E-phi}
\|\mathcal{E}_{f}(\phi)\|^{(-1-\alpha,\Gamma_{\rm w})}_{2,\alpha,\mathcal{N}_{-1/2}^+}\le C\|\phi\|^{(-1-\alpha,\Gamma_{{\rm w},f}^+)}_{2,\alpha,\mathcal{N}_{f}^+}.
\end{equation}
Then, it follows from \eqref{phi-est2} and \eqref{est-E-phi} that 
\begin{equation*}
\|\mathcal{E}_{f}(\phi)\|_{2,\alpha,\mathcal{N}_{-1/2}^+}^{(-1-\alpha,\Gamma_{\rm w})}
\le  C_{\star}\left(\sigma+(M_1\sigma)^2+M_2\sigma+M_3\sigma\right)
\end{equation*}
for a constant $C_{\star}>0$ depending only on the data.

{\bf 5.} 
For fixed $(S_{\ast},\Lambda_{\ast},\psi_{\ast})\in\mathcal{I}(M_2)\times\mathcal{I}(M_3)$, define an iteration mapping $\mathcal{J}^{(S_{\ast},\Lambda_{\ast},\psi_{\ast})}: \mathcal{I}(M_1)\longrightarrow C_{(-1-\alpha,\Gamma_{\rm w})}^{2,\alpha}(\mathcal{N}_{-1/2}^+)$ by 
\begin{equation*}
\mathcal{J}^{(S_{\ast},\Lambda_{\ast},\psi_{\ast})}(\phi_{\ast})=\mathcal{E}_{f}(\phi),
\end{equation*}
where $\phi$ is the solution to \eqref{phi-equ} associated with $\phi_{\ast}$.
We choose constants $M_1$ and $\sigma_4^{\ast}$ as
\begin{equation}\label{sigma4star}
M_1:=4C_{\star}+4C_{\star}M_2+4C_{\star}M_3\quad\mbox{and}\quad\sigma_4^{\ast}:=\min\left\{\sigma_3,\frac{1}{4C_{\star}M_1},\frac{\epsilon_1}{M_1+M_2+M_3}\right\}
\end{equation}
with $\sigma_3$ defined in \eqref{sigma-free} so that 
the mapping $\mathcal{J}^{(S_{\ast},\Lambda_{\ast},\psi_{\ast})}$ maps $\mathcal{I}(M_1)$ into itself whenever $\sigma\le\sigma_4^{\ast}$.

\sloppy The iteration set $\mathcal{I}(M_1)$ given in \eqref{Ite-set-phi} is a convex and compact subset of $C_{(-1-\alpha/2,\Gamma_{\rm w})}^{2,\alpha/2}(\mathcal{N}_{-1/2}^+)$. By using the uniqueness of a solution of \eqref{phi-equ}, one can show that the iteration mapping $\mathcal{J}^{(S_{\ast},\Lambda_{\ast},\psi_{\ast})}$ is continuous in  $C_{(-1-\alpha/2,\Gamma_{\rm w})}^{2,\alpha/2}(\mathcal{N}_{-1/2}^+)$.
Then the Schauder fixed point theorem implies that $\mathcal{J}^{(S_{\ast},\Lambda_{\ast},\psi_{\ast})}$ has a fixed point $\phi_{\sharp}\in\mathcal{I}(M_1)$. 
For such $\phi_{\sharp}$, there exists a radial function $f:B_1(0)\longrightarrow(-\frac{1}{4},\frac{1}{4})$ satisfying
\begin{equation}\label{free}
f(r)=\frac{\phi_{\sharp}-(\varphi^--\varphi_0^-)}{u_0^--u_0^+}(f(r),r)
\end{equation}
by Lemma \ref{shock-lemma}. 
Then $(f,\left.\phi_{\sharp}\right|_{\mathcal{N}_{f}^+}+\varphi_0^+)$ solves Problem \ref{free-Pro} and satisfies the estimate \eqref{lem2-est} for $\sigma\le\sigma_4^{\ast}$.

Let $(f^{(1)},\varphi^{(1)})$ and $(f^{(2)},\varphi^{(2)})$ be two solutions of Problem \ref{free-Pro}, and suppose that each solution satisfies the estimate \eqref{lem2-est}. 
Set
\begin{equation*}
\left\{\begin{split}
&\widetilde{\phi}:=(\varphi^{(1)}-\varphi_0^+)-\left((\varphi^{(2)}-\varphi_0^+)\circ\mathfrak{T}\right),\quad\widetilde{f}:=f^{(1)}-f^{(2)},\\
&\widetilde{\psi}:=\psi_{\ast}-\left(\psi_{\ast}\circ\mathfrak{T}\right),\quad\widetilde{S}:=S_{\ast}-\left(S_{\ast}\circ\mathfrak{T}\right),\quad\widetilde{\Lambda}:=\Lambda_{\ast}-\left(\Lambda_{\ast}\circ\mathfrak{T}\right)
\end{split}\right.
\end{equation*}
for an invertible transformation $\mathfrak{T}:\overline{\mathcal{N}_{f^{(1)}}^+}\longrightarrow\overline{\mathcal{N}_{f^{(2)}}^+}$ defined by 
\begin{equation}\label{trans1-2}
\mathfrak{T}(x_1,x_2,x_3):=\left(\frac{1-f^{(2)}(r)}{1-f^{(1)}(r)}(x_1-1)+1,x_2,x_3\right).
\end{equation}
Then, a direct computation with \eqref{free} yields that
\begin{equation}\label{f-diff-est}
\begin{split}
&\|\widetilde{f}\|^{(-\alpha,\partial B_1(0))}_{1,\alpha,B_1(0)}\le C\|\widetilde{\phi}\|^{(-\alpha,\Gamma_1)}_{1,\alpha,\mathcal{N}_1}\quad\mbox{for}\quad\sigma\le \frac{|u_0^+-u_0^-|}{3},\\
&\|\widetilde{\psi}\|^{(-\alpha,D_1)}_{1,\alpha,\Omega_1}+\|(\widetilde{S},\widetilde{\Lambda})\|^{(1-\alpha,\Gamma_1)}_{0,\alpha,\mathcal{N}_1}\le C(M_2+M_3)\sigma\|\widetilde{f}\|_{1,\alpha,B_1(0)}^{(-\alpha,\partial B_1(0))},
\end{split}
\end{equation}
where we set ($\mathcal{N}_1, \Omega_1,\Gamma_1,D_1$) as 
\begin{equation*}\label{simple-no}
\begin{split}
&\mathcal{N}_1:=\mathcal{N}_{f^{(1)}}^+,\quad\Gamma_1:=\partial\mathcal{N}_1\cap\{r=1\},\\
&\Omega_1:=\{(x,r)\in\mathbb{R}^2:\,f^{(1)}(r)<x<1, 0<r<1\},\quad D_1:=\partial\Omega_1\cap\{r=1\}.
\end{split}
\end{equation*}
To prove the uniqueness of a solution, 
we  estimate $\widetilde{\phi}$ in $\mathcal{N}_{f^{(1)}}^+$. 
For that purpose, we first get a boundary value problem for $\widetilde{\phi}$ in $\mathcal{N}_{f^{(1)}}^+$ by subtracting the boundary value problem  \eqref{phi-equ} for $\left(\varphi^{(2)}-\varphi_0^+\right)\circ\mathfrak{T}$ in  $\mathcal{N}_{f^{(1)}}^+$ from the boundary value problem \eqref{phi-equ} for $\varphi^{(1)}-\varphi_0^+$ in $\mathcal{N}_{f^{(1)}}^+$. 
 By using  \eqref{f-diff-est} and similar methods used to obtain the estimate \eqref{phi-est-lin}, we get
\begin{equation}\label{phi-est-uni}
\|\widetilde{\phi}\|^{(-\alpha,\Gamma_1)}_{1,\alpha,\mathcal{N}_1}\le C_{\star\star}(1+M_2+M_3)\sigma\|\widetilde{\phi}\|^{(-\alpha,\Gamma_1)}_{1,\alpha,\mathcal{N}_1}
\end{equation}
for a constant $C_{\star\star}>0$ depending only on the data.
Finally, we choose $\sigma_4$ to be 
\begin{equation}\label{sigma4}
\sigma_4:=\min\left\{ \sigma_4^{\ast}, \frac{|u_0^+-u_0^-|}{3}, \frac{1}{2C_{\star\star}(1+M_2+M_3)} \right\}
\end{equation}
for $\sigma_4^{\ast}$ defined in \eqref{sigma4star} so that \eqref{f-diff-est}-\eqref{phi-est-uni} imply that $(f^{(1)},\varphi^{(1)})=(f^{(2)},\varphi^{(2)})$ for $\sigma\le\sigma_4$. 
This finishes the proof of Lemma \ref{lemma-free}. 
\end{proof}

\subsection{Free boundary problem for $(\varphi,S,\Lambda)$}\label{sec-lem2}

\begin{problem}\label{Pro2}
For each $\psi_{\ast}\in\mathcal{I}(M_3)$, set
\begin{equation*}
{\bf q}^{\star}:={\bf q}(r,\psi_{\ast},D\psi_{\ast},D\varphi,\Lambda),\quad {\bf t}^{\star}:={\bf t}(r,\psi_{\ast},D\psi_{\ast},\Lambda)
\end{equation*}
for $({\bf q},{\bf t})$ given in \eqref{3D-u} and \eqref{def-T}, respectively. Find $(f,\varphi,S,\Lambda)$ satisfying 
\begin{equation*}
\left\{\begin{split}
	H(S,{\bf q}^{\star}){\bf q}^{\star}\cdot\nabla (S,\Lambda)={\bf 0}\quad&\mbox{in}\quad\mathcal{N}_{f}^+,\\
	(S,\Lambda)=\left(S_{\rm sh}(f'),\Lambda^-\right)\quad&\mbox{on}\quad\mathfrak{S}_{f},
\end{split}\right.
\end{equation*}
and
\begin{equation}\label{FixS-varphi}
\left\{\begin{split}
	\mbox{div\,}(H(S,{\bf q}^{\star}){\bf q}^{\star})=0\quad&\mbox{in}\quad\mathcal{N}_{f}^+,\\
	\varphi=\varphi^-,\quad-\nabla\varphi\cdot{\bf n}_{f}=-\frac{K_s(f')}{{\bf u}^-\cdot{\bf n}_f}+{\bf t}^{\star}\cdot{\bf n}_f\quad&\mbox{on}\quad\mathfrak{S}_{f},\\
\partial_r\varphi=0\quad&\mbox{on}\quad\Gamma_{{\rm w},f}^+,\\
\varphi=\varphi_0^+\quad&\mbox{on}\quad\Gamma_{\rm ex},
\end{split}\right.
\end{equation}
where $H$, $S_{\rm sh}$, and $K_s(f')$ are given in \eqref{def-H-G}, \eqref{def-A}, and  \eqref{def-KS}, respectively.
\end{problem}

\begin{lemma}\label{lemma2}
There exists a small constant $\sigma_5>0$ depending only on the data and $ M_3$ so that if 
$$\sigma\le\sigma_5,$$
then, for each $\psi_{\ast}\in\mathcal{I}(M_3)$,  Problem \ref{Pro2}  has a unique axisymmetric solution $(f,\varphi,S,\Lambda)$ that satisfies
\begin{equation}\label{lem3-est}
\begin{split}
\|f\|_{2,\alpha,B_1(0)}^{(-1-\alpha,\partial B_1(0))}+\|\varphi-\varphi_0^+\|^{(-1-\alpha,\Gamma_{{\rm w},f}^+)}_{2,\alpha,\mathcal{N}_f^+}&\le C(1+M_3)\sigma,\\
\|(S,\Lambda)-(S_0^+,0)\|^{(-\alpha,\Gamma_{{\rm w},f}^+)}_{1,\alpha,\mathcal{N}_f^+}&\le C(1+M_3)\sigma,
\end{split}
\end{equation}
where the constant $C>0$ depends only on the data.
\end{lemma}

\begin{proof} The proof of Lemma \ref{lemma2} is divided into three steps.

{\bf 1.}  Fix $(S_{\ast},\Lambda_{\ast})\in\mathcal{I}(M_2)$. By Lemma \ref{lemma-free}, if $\sigma\le\sigma_4$, then there exists a unique axisymmetric solution $(f,\varphi)$ satisfying \eqref{FixS-varphi} associated with $(S,\Lambda)=(S_{\ast},\Lambda_{\ast})$.
Moreover, the solution $(f,\varphi)$ satisfies
\begin{equation}\label{est-f-var}
\|f\|_{2,\alpha,B_1(0)}^{(-1-\alpha,\partial B_1(0))}+\|\varphi-\varphi_0^+\|^{(-1-\alpha,\Gamma_{{\rm w},f}^+)}_{2,\alpha,\mathcal{N}_f^+}\le C(1+M_2+M_3)\sigma.
\end{equation}

{\bf 2.} (Initial value problem for $(S,\Lambda)$)
We find a solution $(S,\Lambda)$ of the following initial value problem:
\begin{equation}\label{trans}
\left\{\begin{split}
	H(S_{\ast},{\bf q}_{\star}){\bf q}_{\star}\cdot\nabla (S,\Lambda)={\bf 0}\quad&\mbox{in}\quad\mathcal{N}_{f}^+,\\
	(S,\Lambda)=\left(S_{\rm sh}(f'),\Lambda^-\right)\quad&\mbox{on}\quad\mathfrak{S}_{f},
\end{split}\right.
\end{equation}
with 
${\bf q}_{\star}:={\bf q}(r,\psi_{\ast},D\psi_{\ast},D\varphi,\Lambda_{\ast}).$
To solve this problem, we apply the proof of \cite[Proposition 3.5]{bae20183}.
For that purpose, we first rewrite \eqref{trans} as a problem defined in $\mathcal{N}_0^+:=\mathcal{N}\cap\{x>0\}$ by using the change of variables with a flattening map $\mathfrak{T}_f:\mathcal{N}_f^+\longrightarrow\mathcal{N}_0^+$ defined by 
\begin{equation*}
\mathfrak{T}_f(x_1,x_2,x_3)=\left(\frac{x_1-1}{1-f(r)}+1,x_2,x_3\right).
\end{equation*}
Since $|f|\le\frac{1}{4}$, $\mathfrak{T}_f$ is invertible and 
\begin{equation*}
\mathfrak{T}^{-1}_f(y_1,y_2,y_3)=((y_1-1)(1-f(t))+1,y_2,y_3),
\end{equation*}
where $(y,t,\theta)$ denote the cylindrical coordinates of $(y_1,y_2,y_3)\in\mathcal{N}_0^+$, i.e., 
$$(y_1,y_2,y_3)=(y,t\cos\theta,t\sin\theta),\quad t\ge0,\quad\theta\in\mathbb{T}.$$
For ${\bf y}\in\mathcal{N}_0^+$ and ${\bf y}_0\in\partial\mathcal{N}_0^+\cap\{y=0\}$, set
\begin{equation}\label{S-tilde}
\begin{split}
&({S}^{\star},{\Lambda}^{\star})({\bf y}):=(S,\Lambda)\circ\mathfrak{T}_f^{-1}({\bf y}),\quad(\widetilde{S_{\rm sh}},\widetilde{\Lambda_{\rm sh}})({\bf y}_0):=(S_{\rm sh}(f'),\Lambda^-)\circ\mathfrak{T}_f^{-1}({\bf y}_0),\\
&M_x({\bf y}):=\left(H(S_{\ast},{\bf q}_{\star}){\bf q}_{\star}\cdot{\bf e}_x\right)\circ\mathfrak{T}^{-1}_f({\bf y}),\quad M_r({\bf y}):=\left(H(S_{\ast},{\bf q}_{\star}){\bf q}_{\star}\cdot{\bf e}_r\right)\circ\mathfrak{T}_f^{-1}({\bf y}),
\end{split}
\end{equation}
to rewrite the initial value problem \eqref{trans} as follows:
\begin{equation}\label{New-trans}
\left\{\begin{split}
(N_y\partial_y+N_t\partial_t)({S}^{\star},{\Lambda}^{\star})={\bf 0}\quad&\mbox{in}\quad\mathcal{N}_0^+,\\
({S}^{\star},{\Lambda}^{\star})=(\widetilde{S_{\rm sh}},\widetilde{\Lambda_{\rm sh}})\quad&\mbox{on}\quad\partial\mathcal{N}_0^+\cap\{y=0\},
\end{split}\right.
\end{equation}
with 
\begin{equation*}
N_y:=M_x+(y-1)f'(t)M_r\quad\mbox{and}\quad N_t:=(1-f(t))M_r.
\end{equation*}

In the below (i)-(iii), we check that the sufficient conditions to apply \cite[Proposition 3.5]{bae20183} are hold:\\
(i) From the equation $\mbox{div\,}(H(S_{\ast},{\bf q}_{\star}){\bf q}_{\star})=0$ in $\mathcal{N}_f^+$, one can directly check that 
\begin{equation}\label{div-trans}
\partial_y(tN_y)+\partial_t(tN_t)=0\quad\mbox{in}\quad\mathcal{N}_0^+.
\end{equation}
(ii) Since $(S_{\ast},\Lambda_{\ast},\psi_{\ast})\in\mathcal{I}(M_2)\times\mathcal{I}(M_3)$ and the estimate \eqref{est-f-var} holds, there exists a constant $\epsilon_2>0$ depending only on the data so that if 
\begin{equation*}
(M_2+M_3)\sigma\le \epsilon_2,
\end{equation*}
then we have
\begin{equation*}
\|N_y-\rho_0^+u_0^+\|_{0,\mathcal{N}_0^+}\le C_{\sharp}(1+M_2+M_3)\sigma
\end{equation*}
for a constant $C_{\sharp}>0$ depending only on the data.
Thus
\begin{equation}\label{Ny-est}
0<\frac{\rho_0^+u_0^+}{2}\le N_y\le\frac{3\rho_0^+u_0^+}{2}\quad\mbox{in}\quad\overline{\mathcal{N}_0^+}
\end{equation}
for
\begin{equation}\label{tilde-sigma5}
\sigma\le\min\left\{\frac{\epsilon_2}{M_2+M_3},\frac{\rho_0^+u_0^+}{2C_{\sharp}(1+M_2+M_3)}\right\}=:\widetilde{\sigma_5}.
\end{equation}
(iii) The conditions
\begin{equation*}
\partial_r\varphi=0\quad\mbox{and}\quad \psi_{\ast}=0\quad\mbox{on}\quad \Gamma_{{\rm w},f}^+\cup\left(\mathcal{N}_f^+\cup\{r=0\}\right)
\end{equation*}
imply that
\begin{equation*}
N_t\equiv 0\quad\mbox{on}\quad \left(\partial\mathcal{N}_0^+\cap\{t=1\}\right)\cup\left(\mathcal{N}_0^+\cap\{t=0\}\right).
\end{equation*}

Now we regard \eqref{New-trans} as a problem defined in a two-dimensional rectangular domain 
$\Omega^+:=\left\{(y,t)\in\mathbb{R}^2:\, 0<y<1,\, 0<t<1\right\}$, and apply \cite[Proposition 3.5]{bae20183}.
Then the initial value problem \eqref{New-trans} has the unique solution $({S}^{\star},{\Lambda}^{\star})$ defined by 
\begin{equation*}\label{sol-trans}
({S}^{\star},{\Lambda}^{\star})(y,t):=(\widetilde{S_{\rm sh}},\widetilde{\Lambda_{\rm sh}})(\mathcal{R}_{\rm sh}(y,t))\quad\mbox{for}\quad(y,t)\in\overline{\Omega^+}.
\end{equation*}
Here, $\mathcal{R}_{\rm sh}:\overline{\Omega^+}\longrightarrow[0,1]$ is a function defined by 
\begin{equation}\label{def-Rsh}
\mathcal{R}_{\rm sh}(y,t):=\mathcal{G}^{-1}\circ w(y,t)
\end{equation}
for $w:\overline{\Omega^+}\longrightarrow\mathbb{R}^+$ given by 
\begin{equation}\label{def-w}
w(y,t):=\int_0^t zN_y(y,z)dz
\end{equation}
and  an invertible function $\mathcal{G}:[0,1]\longrightarrow[w(0,0),w(0,1)]$ given by
\begin{equation*}
\mathcal{G}(r):=w(0,r).
\end{equation*} 
By \eqref{Ny-est} and the definition of $\mathcal{R}_{\rm sh}$,
there exists a constant $\mu\in(0,1)$ such that 
\begin{equation}\label{Rsh-est}
\mu\le\frac{\mathcal{R}_{\rm sh}(y,1)-\mathcal{R}_{\rm sh}(y,t)}{1-t}=\frac{1-\mathcal{R}_{\rm sh}(y,t)}{1-t}\le\frac{1}{\mu}.
\end{equation}
By \eqref{Rsh-est} and \cite[Proposition 3.5]{bae20183}, the solution $(S^{\star},\Lambda^{\star})$ satisfies 
\begin{equation*}\label{trans-esti}
\|({S}^{\star},{\Lambda}^{\star})-(S_0^+,0)\|^{(-\alpha,\{t=1\})}_{1,\alpha,\Omega^+}\le C\|(\widetilde{S_{\rm sh}},\widetilde{\Lambda_{\rm sh}})-(S_0^+,0)\|_{1,\alpha,\partial\Omega^+\cap\{y=0\}}^{(-\alpha,\{t=1\})}.
\end{equation*}
Then  $(S,\Lambda):=({S}^{\star},{\Lambda}^{\star})\circ\mathfrak{T}_f$ is the unique solution of the initial value problem \eqref{trans} and the solution satisfies
\begin{equation}\label{S-L-est}
\|(S,\Lambda)-(S_0^+,0)\|^{(-\alpha,\Gamma_{{\rm w},f}^+)}_{1,\alpha,\mathcal{N}_f^+}\le C\|(S_{\rm sh}(f'),\Lambda^-)-(S_0^+,0)\|_{1,\alpha,\mathfrak{S}_f}^{(-\alpha,\partial\mathfrak{S}_f)}.
\end{equation}

{\bf 3.} 
For a fixed $\psi_{\ast}\in\mathcal{I}(M_3)$, define an iteration mapping $\mathcal{J}^{\psi_{\ast}}:\mathcal{I}(M_2)\longrightarrow [C^{1,\alpha}_{(-\alpha,\Gamma_{\rm w})}(\mathcal{N}_{-1/2}^+)]^2$ by 
\begin{equation*}
\mathcal{J}^{\psi_{\ast}}(S_{\ast},\Lambda_{\ast})=(\mathcal{E}_f(S),\mathcal{E}_f(\Lambda)),\end{equation*}
where $\mathcal{E}_f$ is given by \eqref{extension} and $(S,\Lambda)$ is the solution to \eqref{trans} associated with $(S_{\ast},\Lambda_{\ast})$.
A direct computation with using \eqref{S-L-est} yields that 
\begin{equation*}
\|(\mathcal{E}_f(S),\mathcal{E}_f(\Lambda))-(S_0^+,0)\|_{1,\alpha,\mathcal{N}_{-1/2}^+}^{(-\alpha,\Gamma_{\rm w})}
\le C_{\sharp\sharp}\left(\sigma+(M_1\sigma)^2\right)
\end{equation*}
for $M_1>0$ given by \eqref{sigma4star} and $C_{\sharp\sharp}>0$ depending only on the data.

For further estimate, set $\mathcal{V}$ as
\begin{equation*}
\mathcal{V}(x,r):=
	\frac{\mathcal{E}_f(\Lambda)(x,r)}{r}\quad\mbox{for}\quad(x,r)\in\Omega_{-1/2}^+.
\end{equation*}
Then one can directly check from \eqref{div-trans}, \eqref{def-Rsh} and \eqref{def-w} that 
\begin{equation*}
\lim_{r\rightarrow0+}\partial_r\mathcal{V}=\lim_{r\rightarrow0+}\partial_r\left(\frac{\Lambda^{\star}\circ\mathfrak{T}}{r}\right)=0\quad\mbox{in}\quad \Omega_f^+:=\Omega_{-1/2}^+\cap\{x\ge f(r)\}.
\end{equation*}
By the definition of $\mathcal{E}_f$ given by \eqref{extension}, $\mathcal{V}$ satisfies
\begin{equation*}
\|\mathcal{V}\|_{1,\alpha,\Omega_{-1/2}^+}^{(-\alpha,\{r=1\})}\le C_{\sharp\sharp\sharp}\sigma
\end{equation*}
for a constant $C_{\sharp\sharp\sharp}>0$ depending only on the data.

Choose $M_2$ and $\sigma_5^{\ast}$ as 
\begin{equation}\label{sigma5star}
M_2:=2\left(C_{\sharp\sharp}+C_{\sharp\sharp\sharp}\right)\quad\mbox{and}\quad \sigma_5^{\ast}:=\min\left\{\sigma_4,\widetilde{\sigma_5},\frac{1}{1+M_2+M_3},\frac{M_2}{2C_{\sharp\sharp}M_1^2}\right\}
\end{equation}
for $(\sigma_4,\widetilde{\sigma_5})$ given in \eqref{sigma4} and \eqref{tilde-sigma5}.
Then, under such choices of $(M_2,\sigma_5^{\ast})$, the mapping $\mathcal{J}^{\psi_{\ast}}$ maps $\mathcal{I}(M_2)$ into itself whenever $\sigma\le\sigma_5^{\ast}$.

 The iteration set $\mathcal{I}(M_2)$ defined in \eqref{Ite-set-T} is a compact and convex subset of $[C^{1,\alpha/2}_{(-\alpha/2,\Gamma_{\rm w})}(\mathcal{N}_{-1/2}^+)]^2$.
By Lemma \ref{lemma-free} and the uniqueness of a solution for \eqref{trans},  one can prove that $\mathcal{J}^{\psi_{\ast}}$ is continuous in $[C^{1,\alpha/2}_{(-\alpha/2,\Gamma_{\rm w})}(\mathcal{N}_{-1/2}^+)]^2$.
Then the Schauder fixed point theorem yields that there exists a fixed point of $\mathcal{J}^{\psi_{\ast}}$, say $(S_{\sharp},\Lambda_{\sharp})$.
By Lemma \ref{lemma-free}, there exists a solution $(f,\varphi)$ of the free boundary problem \eqref{varphi-eq} associated with $(S_{\ast},\Lambda_{\ast})=(S_{\sharp},\Lambda_{\sharp})$ and the solution satisfies the estimate \eqref{lem2-est}.
Then $(f,\varphi,\left.S_{\sharp}\right|_{\mathcal{N}_f^+},\left.\Lambda_{\sharp}\right|_{\mathcal{N}_f^+})$ solves Problem \ref{Pro2} and satisfies the estimate \eqref{lem3-est} for $\sigma\le\sigma_5^{\ast}$.

Let $\mathcal{U}^{(k)}:=(f^{(k)},\varphi^{(k)},S^{(k)},\Lambda^{(k)})$ ($k=1,2$) be two solutions to Problem \ref{Pro2}, and suppose that each solution satisfies the estimate \eqref{lem3-est}.
Set
\begin{equation*}
\left\{\begin{split}
&\widetilde{\psi}:=\psi_{\ast}-\left(\psi_{\ast}\circ\mathfrak{T}\right),\quad \widetilde{f}:=f^{(1)}-f^{(2)},\\
&\widetilde{\phi}:=(\varphi^{(1)}-\varphi_0^+)-\left((\varphi^{(2)}-\varphi_0^+)\circ\mathfrak{T}\right),\\
&\widetilde{S}:=S^{(1)}-\left(S^{(2)}\circ\mathfrak{T}\right),\quad\widetilde{\Lambda}:=\Lambda^{(1)}-\left(\Lambda^{(2)}\circ\mathfrak{T}\right)
\end{split}\right.
\end{equation*}
for a transformation $\mathfrak{T}:\overline{\mathcal{N}_{f^{(1)}}^+}\longrightarrow\overline{\mathcal{N}_{f^{(2)}}^+}$ defined in \eqref{trans1-2}.
For simplicity of notations, let $(\mathcal{N}_k,  \Omega_k,\Gamma_k,D_k)$ $(k=1,2)$ denote 
\begin{equation*}
\begin{split}
&\mathcal{N}_k:=\mathcal{N}_{f^{(k)}}^+,\quad \Gamma_k:=\partial\mathcal{N}_k\cap\{r=1\},\\
&\Omega_k:=\{(x,r)\in\mathbb{R}^2:f^{(k)}(r)<x<1,\,0<r<1\},\quad D_k:=\partial\Omega_k\cap\{r=1\}.
\end{split}
\end{equation*}
By a method similar to the proof of Lemma \ref{lemma-free} with 
\begin{equation}\label{f-est-22}
\|\widetilde{f}\|^{(-\alpha,\partial B_1(0))}_{1,\alpha,B_1(0)}\le C\|\widetilde{\phi}\|_{1,\alpha,\mathcal{N}_1}^{(-\alpha,\Gamma_1)}\quad\mbox{and}\quad
\|\widetilde{\psi}\|_{1,\alpha,\Omega_1}^{(-\alpha,D_1)}\le CM_3\sigma\|\widetilde{f}\|_{1,\alpha,B_1(0)}^{(-\alpha,\partial B_1(0))},\end{equation}
one can show that 
\begin{equation*}
\|\widetilde{\phi}\|^{(-\alpha,\Gamma_1)}_{1,\alpha,\mathcal{N}_1}\le C^{\ast}(1+M_3)\left(\sigma\|\widetilde{\phi}\|^{(-\alpha,\Gamma_1)}_{1,\alpha,\mathcal{N}_1}+\|(\widetilde{S},\widetilde{\Lambda})\|_{0,\alpha,\mathcal{N}_1}^{(1-\alpha,\Gamma_1)}\right)
\end{equation*}
for a constant $C^{\ast}>0$ depending only on the data.
If it holds that 
\begin{equation*}
\sigma\le\frac{1}{2C^{\ast}(1+M_3)},
\end{equation*}
then
\begin{equation}\label{phi-diff-lem2}
\|\widetilde{\phi}\|^{(-\alpha,\Gamma_1)}_{1,\alpha,\mathcal{N}_1}\le C^{\ast}(1+M_3)\|(\widetilde{S},\widetilde{\Lambda})\|_{0,\alpha,\mathcal{N}_1}^{(1-\alpha,\Gamma_1)}.
\end{equation}
For $k=1,2$, let $\widetilde{S_{\rm sh}^{(k)}}$ and $\mathcal{R}_{\rm sh}^{(k)}$ $(k=1,2)$ be defined in \eqref{S-tilde} and \eqref{def-Rsh} associated with $\mathcal{U}^{(k)}$.
Then a direct computation with using \eqref{lem3-est} and \eqref{f-est-22}-\eqref{phi-diff-lem2} yields that 
\begin{equation*}\label{SS-R}
\begin{split}
\|\widetilde{S_{\rm sh}^{(1)}}(\mathcal{R}_{\rm sh}^{(1)})-\widetilde{S_{\rm sh}^{(2)}}(\mathcal{R}_{\rm sh}^{(2)}))\|^{(1-\alpha,\{t=1\})}_{0,\alpha,\mathcal{N}_0}&\le C(1+M_3)\sigma\|\mathcal{R}_{\rm sh}^{(1)}-\mathcal{R}_{\rm sh}^{(2)}\|^{(1-\alpha,\{r=1\})}_{0,\alpha,\Omega^+}\\
&\le C(1+M_3)\sigma\|(\widetilde{S},\widetilde{\Lambda})\|_{0,\alpha,\mathcal{N}_1}^{(1-\alpha,\Gamma_1)},
\end{split}
\end{equation*}
and 
\begin{equation}\label{phi-est-final}
\|(\widetilde{S},\widetilde{\Lambda})\|^{(1-\alpha,\Gamma_1)}_{0,\alpha,\mathcal{N}_1}\le C^{\ast\ast}(1+M_3)\sigma\|(\widetilde{S},\widetilde{\Lambda})\|_{0,\alpha,\mathcal{N}_1}^{(1-\alpha,\Gamma_1)}
\end{equation}
for a constant $C^{\ast\ast}>0$ depending only on the data. 
Finally, we choose $\sigma_5$ as 
\begin{equation}\label{sigma5}
\sigma_5:=\min\left\{\sigma_5^{\ast},\frac{1}{2C^{\ast}(1+M_3)},\frac{1}{2C^{\ast\ast}(1+M_3)}\right\}
\end{equation}
for $\sigma_5^{\ast}$ defined in \eqref{sigma5star} so that \eqref{f-est-22}-\eqref{phi-diff-lem2} and \eqref{phi-est-final} imply that 
$(f^{(1)},\varphi^{(1)}, S^{(1)},\Lambda^{(1)})=(f^{(2)},\varphi^{(2)},S^{(2)},\Lambda^{(2)})$ for $\sigma\le\sigma_5$. 
This completes the proof of Lemma \ref{lemma2}.
\end{proof}

\subsection{Proof of Theorem \ref{Thm-HD}}\label{sec-Thm-hd}
The proof of Theorem \ref{Thm-HD} is divided into three steps. 

{\bf 1.} Fix $\psi_{\ast}\in\mathcal{I}(M_3)$. According to Lemma \ref{lemma2}, there exists a unique solution $(f,\varphi,S,\Lambda)$ of Problem \ref{Pro2} for $\sigma\le\sigma_5$, and the solution satisfies the estimate \eqref{lem3-est}.
For such $(f,\varphi,S,\Lambda)$, we consider the following boundary value problem for ${\bf W}$: 
\begin{equation}\label{psi-prob}
\left\{\begin{split}
-\Delta{\bf W}=G_{\flat}{\bf e}_{\theta}\quad&\mbox{in}\quad\mathcal{N}_{f}^+,\\
-\nabla{\bf W}\cdot{\bf n}_{f}=\mathcal{A}_{\flat}{\bf e}_{\theta}\quad&\mbox{on}\quad\mathfrak{S}_{f},\\
{\bf W}={\bf 0}\quad&\mbox{on}\quad\Gamma_{{\rm w},f}^+,\\
\partial_x{\bf W}={\bf 0}\quad&\mbox{on}\quad\Gamma_{\rm ex},
\end{split}\right.
\end{equation}
with
\begin{equation*}
{\bf t}_{\flat}:={\bf t}(f,\psi_{\ast},D\psi_{\ast},\Lambda),\quad G_{\flat}:=G(S,\Lambda,\partial_r S,\partial_r\Lambda,{\bf t}_{\flat},\nabla\varphi),\quad
\mathcal{A}_{\flat}:=\mathcal{A}(r,\psi_{\ast},f')
\end{equation*}
for $({\bf t},G,\mathcal{A})$ given by \eqref{def-T}, \eqref{def-H-G}, and \eqref{def-A}, respectively.


Let $\mathcal{H}:=\left\{\zeta\in H^1(\mathcal{N}_f^+):\,\zeta=0\mbox{ on }\Gamma_{{\rm w},f}^+\right\}$. 
For $k=1,2,3$, if each $W_k\in\mathcal{H}$ satisfies 
\begin{equation*}
\mathfrak{L}[W_k,\zeta]=\langle (G_{\flat}{\bf e}_{\theta}\cdot{\bf e}_k),(\mathcal{A}_{\flat}{\bf e}_{\theta}\cdot{\bf e}_k),\zeta\rangle
\end{equation*}
for all $\zeta\in\mathcal{H}$, where 
\begin{equation*}
\begin{split}
&\mathfrak{L}[W_k, \zeta]:=\int_{\mathcal{N}_f^+}\sum_{i=1}^3 (\partial_iW_k)(\partial_i\zeta)\,d{\bf x},\\
&\langle (G_{\flat}{\bf e}_{\theta}\cdot{\bf e}_k),(\mathcal{A}_{\flat}{\bf e}_{\theta}\cdot{\bf e}_k),\zeta\rangle:=\int_{\mathcal{N}_f^+}(G_{\flat}{\bf e}_{\theta}\cdot{\bf e}_k)\zeta\, d{\bf x}+\int_{\mathfrak{S}_f}(\mathcal{A}_{\flat}{\bf e}_{\theta}\cdot{\bf e}_k)\zeta\,dS,
\end{split}
\end{equation*}
then we call ${\bf W}=(W_1,W_2,W_3)$ a weak solution of the problem \eqref{psi-prob}.

A direct computation implies that  
\begin{equation*}
\mathfrak{L}[W_k,\zeta]\le C\|W_k\|_{H^1(\mathcal{N}_f^+)}\|\zeta\|_{H^1(\mathcal{N}_f^+)},
\end{equation*}
and the Poincar\'e inequality yields that there exists $\nu_0>0$ depending only on the data such that 
\begin{equation*}
\mathfrak{L}[\zeta,\zeta]\ge \nu_0\|\zeta\|^2_{H^1(\mathcal{N}_f^+)}\quad\mbox{for all }\zeta\in\mathcal{H}.
\end{equation*}
Thus $\mathfrak{L}$ is a bounded bilinear functional on $\mathcal{H}\times\mathcal{H}$ and coercive. 
Since $G_{\flat}{\bf e}_{\theta}\in C_{(1-\alpha,\Gamma^+_{{\rm w},f})}^{0,\alpha}(\mathcal{N}_f^+)$, the definition of the weighted H\"older norms yields that  
\begin{equation*}\label{G-wei-est}
|G_{\flat}{\bf e}_{\theta}({\bf x})|\le \delta_{\bf x}^{-1+\alpha}\|G_{\flat}{\bf e}_{\theta}\|_{0,\alpha,\mathcal{N}_f^+}^{(1-\alpha,\Gamma_{{\rm w},f}^+)}\quad\mbox{for}\quad {\bf x}\in\mathcal{N}_f^+,\quad\delta_{\bf x}=\inf_{{\bf z}\in\Gamma_{{\rm w},f}^+}|{\bf x}-{\bf z}|,
\end{equation*}
from which we obtain that 
\begin{equation}\label{lax}
\int_{\mathcal{N}_f^+}|G_{\flat}{\bf e}_{\theta}|^{6/5}d{\bf x}\le C\left(\|G_{\flat}{\bf e}_\theta\|_{0,\alpha,\mathcal{N}_f^+}^{(1-\alpha,\Gamma_{{\rm w},f}^+)}\right)^{6/5}\quad\mbox{for}\quad \alpha\in(\frac{1}{6},1).
\end{equation}
By the H\"older inequality, Sobolev inequality, trace inequality, Poincar\'e inequality, and \eqref{lax}, we have
\begin{equation*}
|\langle (G_{\flat}{\bf e}_{\theta}\cdot{\bf e}_k),(\mathcal{A}_{\flat}{\bf e}_{\theta}\cdot{\bf e}_k),\zeta\rangle|\le C\left(\|G_{\flat}{\bf e}_{\theta}\|_{0,\alpha,\mathcal{N}_f^+}^{(1-\alpha,\Gamma_{{\rm w},f}^+)}+\|\mathcal{A}_{\flat}{\bf e}_{\theta}\|_{0,\alpha,\mathfrak{S}_f}\right)\|\zeta\|_{H^1(\mathcal{N}_f^+)}
\end{equation*}
for all $\zeta\in\mathcal{H}$.
Then the Lax-Milgram theorem yields that the boundary value problem \eqref{psi-prob} has a unique weak solution ${\bf W}\in H^1(\mathcal{N}_f^+)$ satisfying 
\begin{equation*}
\|{\bf W}\|_{H^1(\mathcal{N}_f^+)}\le C\left(\|G_{\flat}{\bf e}_\theta\|_{0,\alpha,\mathcal{N}_f^+}^{(1-\alpha,\Gamma_{{\rm w},f}^+)}+\|\mathcal{A}_{\flat}{\bf e}_{\theta}\|_{0,\alpha,\mathfrak{S}_f}\right).
\end{equation*}
%
%
By a method similar to the proof of Lemma \ref{lemma-free}, 
one can obtain 
\begin{equation*}
\|{\bf W}\|_{1,\alpha,\mathcal{N}_f^+}\le C\left(\|G_\flat {\bf e}_{\theta}\|^{(1-\alpha,\Gamma_{{\rm w},f}^+)}_{0,\alpha,\mathcal{N}_f^+}+\|\mathcal{A}_{\flat}{\bf e}_{\theta}\|^{(-\alpha,\partial\mathfrak{S}_f)}_{1,\alpha,\mathfrak{S}_f}\right).
\end{equation*}
Then, the Schauder estimate with scaling implies that 
\begin{equation*}
\|{\bf W}\|^{(-1-\alpha,\Gamma_{{\rm w},f}^+)}_{2,\alpha,\mathcal{N}_f^+}\le C\left(\|G_\flat {\bf e}_{\theta}\|^{(1-\alpha,\Gamma_{{\rm w},f}^+)}_{0,\alpha,\mathcal{N}_f^+}+\|\mathcal{A}_{\flat}{\bf e}_{\theta}\|^{(-\alpha,\partial\mathfrak{S}_f)}_{1,\alpha,\mathfrak{S}_f}\right).
\end{equation*}

{\bf 2.} By adjusting the proof of \cite[Proposition 3.3]{bae20183}, one can show that ${\bf W}$ has the form of 
$${\bf W}=\psi{\bf e}_{\theta}$$
and $\psi$ solves the boundary value problem 
\begin{equation*}\label{psi-equation}
\left\{\begin{split}
-\left(\partial_{xx}+\frac{1}{r}\partial_r(r\partial_r)-\frac{1}{r^2}\right)\psi=G_{\flat}\quad&\mbox{in}\quad\mathcal{N}_f^+,\\
-\nabla\psi\cdot{\bf n}_f=\mathcal{A}_{\flat}\quad&\mbox{on}\quad\mathfrak{S}_f,\\
\psi=0\quad&\mbox{on}\quad\Gamma_{{\rm w},f}^+\cup\left(\mathcal{N}_f^+\cap\{r=0\}\right),\\
\partial_x\psi=0\quad&\mbox{on}\quad\Gamma_{\rm ex}.
\end{split}\right.
\end{equation*}
Furthermore, $\psi$ satisfies
\begin{equation}\label{psi-rr-0}
\partial_{rr}\psi\equiv 0\quad\mbox{on}\quad\mathcal{N}_f^+\cap\{r=0\}
\end{equation} 
and
\begin{equation}\label{ppsi-est}
\begin{split}
&\|\psi\|_{1,\alpha,\Omega^+_f}\le C\left(\|G_{\flat}\|^{(1-\alpha,\Gamma_{{\rm w},f}^+)}_{0,\alpha,\mathcal{N}_f^+}+\|\mathcal{A}_{\flat}\|^{(-\alpha,\partial\mathfrak{S}_f)}_{1,\alpha,\mathfrak{S}_f}\right),\\
&\|\psi\|_{2,\alpha,\Omega^+_f}^{(-1-\alpha,\{r=1\})}\le C\left(\|G_{\flat}\|^{(1-\alpha,\Gamma_{{\rm w},f}^+)}_{0,\alpha,\mathcal{N}_f^+}+\|\mathcal{A}_{\flat}\|^{(-\alpha,\partial\mathfrak{S}_f)}_{1,\alpha,\mathfrak{S}_f}\right),\\
\end{split}
\end{equation}
where $\Omega_f^+$ is a two-dimensional space defined by $$\Omega_f^+:=\left\{(x,r)\in\mathbb{R}^2:\, f(r)<x<1,\, 0<r<1\right\}.$$
A direct computation yields that there exists a constant $\epsilon_3>0$ depending only on the data so that if 
$$M_3\sigma\le\epsilon_3,$$
then 
\begin{equation}\label{GA-est}
\begin{split}
&\|G_{\flat}\|^{(1-\alpha,\Gamma_{{\rm w},f}^+)}_{0,\alpha,\mathcal{N}_f^+}\le CM_2\sigma\le C\sigma,\\
&\|\mathcal{A}_{\flat}\|^{(-\alpha,\partial\mathfrak{S}_f)}_{1,\alpha,\mathfrak{S}_f}\le C\left(\sigma+(1+M_1)(1+M_3)\sigma^2\right)\le C\left(\sigma+(1+M_3)^2\sigma^2\right).
\end{split}
\end{equation}
It follows from \eqref{ppsi-est}-\eqref{GA-est} that 
\begin{equation}\label{est-psi-2}
\|\psi\|^{(-1-\alpha,\{r=1\})}_{2,\alpha,\Omega^+_f}\le C\left(\sigma+(1+M_3)^2\sigma^2\right).
\end{equation}

{\bf 3.} Define an iteration map $\mathcal{J}:\mathcal{I}(M_3)\longrightarrow C_{(-1-\alpha,\{r=1\})}^{2,\alpha}(\Omega_{-1/2}^+)$ by 
\begin{equation*}
\mathcal{J}(\psi_{\ast})=\mathcal{E}_f(\psi),
\end{equation*}
where $\mathcal{E}_f$ is defined in \eqref{extension} and $\psi{\bf e}_{\theta}$ is the solution to \eqref{psi-prob} associated with $\psi_{\ast}$.
By \eqref{psi-rr-0} and the condition $f'(0)=0$, we have 
\begin{equation*}
\partial_{rr}\mathcal{E}_f(\psi)\equiv 0\quad\mbox{on}\quad\{r=0\}.
\end{equation*}
By the definition of $\mathcal{E}_f$ and  \eqref{est-psi-2}, we also have
\begin{equation*}
\begin{split}
\|\mathcal{E}_f(\psi)\|_{2,\alpha,\Omega^+_{-1/2}}^{(-1-\alpha,\{r=1\})}
&\le C_{\flat}\left(\sigma+(1+M_3)^2\sigma^2\right)
\end{split}
\end{equation*}
for a constant $C_{\flat}>0$ depending only on the data.
Now, we choose constants $M_3$ and $\sigma_2$ as
\begin{equation*}\label{sigma22}
M_3:=4C_{\flat}\quad\mbox{and}\quad\sigma_2:=\min\left\{\sigma_5,\frac{\epsilon_3}{M_3},\frac{M_3}{4C_{\flat}},\frac{1}{8C_{\flat}},\frac{1}{4C_{\flat}M_3}\right\}
\end{equation*}
for $\sigma_5$ given in \eqref{sigma5}.
Then, under such choices of $(M_3,\sigma_2)$, the mapping $\mathcal{J}$ maps $\mathcal{I}(M_3)$ into itself whenever $\sigma\le\sigma_2$.

 The iteration set $\mathcal{I}(M_3)$ defined in \eqref{Ite-set-psi} is a compact and convex subset of $C_{(-1-\alpha/2,\{r=1\})}^{2,\alpha/2}(\Omega_{-1/2}^+)$.
\sloppy By Lemma \ref{lemma2} and the uniqueness of a solution for the boundary value problem \eqref{psi-prob}, one can prove that $\mathcal{J}$ is continuous in $C_{(-1-\alpha/2,\{r=1\})}^{2,\alpha/2}(\Omega_{-1/2}^+)$.
Then the Schauder fixed point theorem implies that $\mathcal{J}$ has a fixed point $\psi_{\sharp}\in\mathcal{I}(M_3)$.
According to Lemma \ref{lemma2}, there exists a solution $(f,\varphi,S,\Lambda)$ of Problem \ref{Pro2}  associated with $\psi_{\ast}=\psi_{\sharp}$.
Then $(f,\varphi,S,\Lambda,\left.\psi_{\sharp}\right|_{\mathcal{N}_f^+})$ is a solution of the free boundary problem \eqref{3D-H} with \eqref{w-ex-bd} and \eqref{boundary-HD},
and the solution satisfies the estimate \eqref{Thm-HD-est} by \eqref{lem3-est} and \eqref{est-psi-2}.
The proof of Theorem \ref{Thm-HD} is completed.
\qed


\section{Proof of Theorem \ref{MainThm}}\label{sec-proof-main}
According to Theorem \ref{Thm-HD}, for $\sigma\le\sigma_2$, the free boundary problem \eqref{3D-H} with \eqref{w-ex-bd} and \eqref{boundary-HD} has a solution $(f, S,\Lambda, \varphi,\psi)$ that satisfies the estimate \eqref{Thm-HD-est}. 
For such a solution $(f, S,\Lambda, \varphi,\psi)$,  we define $({\bf u},\rho,p)$ by 
\begin{equation*}
\begin{split}
&{\bf u}:=\left(\partial_x\varphi+\frac{1}{r}\partial_r(r\psi)\right){\bf e}_x+(\partial_r\varphi-\partial_x\psi){\bf e}_r+\frac{\Lambda}{r}{\bf e}_{\theta},\\
&\rho:=H(S,{\bf u}),\quad p:=S\rho^{\gamma}\quad\mbox{in}\quad\overline{\mathcal{N}_f^+},
\end{split}
\end{equation*}
where $H$ is given by  \eqref{def-H-G}.
Then, due to the estimate \eqref{Thm-HD-est}, $(f,{\bf u},\rho,p)$ satisfy the estimate \eqref{Thm2.1-uniq-est}.
So one can find a small constant $\sigma_1\in(0,\min\{\sigma_2,\frac{\rho_0^-}{2},\frac{u_0^-}{2}\})$ depending only on the data so that if $\sigma\le\sigma_1$, then $(f,{\bf u},\rho,p)$ satisfy 
\begin{equation*}
\left.\begin{split}
\rho\ge \frac{\rho_0^+}{2}>0,\quad{\bf u}\cdot{\bf e}_x\ge\frac{u_0^+}{2}>0,\quad c^2-|{\bf u}|^2\ge \frac{1}{2}\left((c_0^+)^2-(u_0^+)^2\right)>0 \quad&\mbox{in}\quad\overline{\mathcal{N}_f^+},\\
{\bf u}^-\cdot{\bf n}_f>{\bf u}\cdot{\bf n}_f>0\quad&\mbox{on}\quad\mathfrak{S}_f
\end{split}\right.
\end{equation*}
for $c_0^+:=\sqrt{\frac{\gamma p_0^+}{\rho_0^+}}$ and the unit normal vector field ${\bf n}_f$ on $\mathfrak{S}_f$ pointing toward the interior of $\mathcal{N}_f^+$.
The proof of Theorem \ref{MainThm} is completed.\qed


\vspace{.25in}
\noindent
{\bf Acknowledgements:}
The research of Hyangdong Park was supported in part by the National Research Foundation of Korea (NRF) grant funded by the Korea government (MSIT) (No. 20151009350).
The research of Hyeongyu Ryu was supported in part by  Samsung Science and Technology Foundation under Project Number SSTF-BA1502-02.

\smallskip


\begin{thebibliography}{10}

\bibitem{bae2016subsonic}
{\sc M.~Bae, B.~Duan, and C.~Xie}, {\em Subsonic flow for the multidimensional
  euler--poisson system}, Archive for Rational Mechanics and Analysis, 220
  (2016), pp.~155--191.

\bibitem{bae2011transonic}
{\sc M.~Bae and M.~Feldman}, {\em Transonic shocks in multidimensional
  divergent nozzles}, Archive for rational mechanics and analysis, 201 (2011),
  pp.~777--840.

\bibitem{bae2019contact}
{\sc M.~Bae and H.~Park}, {\em Contact discontinuities for 2-dimensional
  inviscid compressible flows in infinitely long nozzles}, SIAM Journal on
  Mathematical Analysis, 51 (2019), pp.~1730--1760.

\bibitem{bae2019contact3D}
{\sc M.~Bae and H.~Park}, {\em Contact discontinuities for 3-d axisymmetric inviscid compressible flows in
  infinitely long cylinders}, Journal of Differential Equations, 267 (2019),
  pp.~2824--2873.

\bibitem{bae20183}
{\sc M.~Bae and S.~Weng}, {\em 3-d axisymmetric subsonic flows with nonzero
  swirl for the compressible euler--poisson system}, in Annales de l'Institut
  Henri Poincar{\'e} C, Analyse non lin{\'e}aire, vol.~35, Elsevier, 2018,
  pp.~161--186.

\bibitem{chen2003multidimensional}
{\sc G.-Q. Chen and M.~Feldman}, {\em Multidimensional transonic shocks and
  free boundary problems for nonlinear equations of mixed type}, Journal of the
  American Mathematical Society, 16 (2003), pp.~461--494.

\bibitem{chen2004steady}
{\sc G.-Q. Chen and M.~Feldman}, {\em Steady transonic
  shocks and free boundary problems in infinite cylinders for the euler
  equations}, Communications on Pure and Applied Mathematics: A Journal Issued
  by the Courant Institute of Mathematical Sciences, 57 (2004), pp.~310--356.

\bibitem{chen2007existence}
{\sc G.-Q. Chen and M.~Feldman}, {\em Existence and
  stability of multidimensional transonic flows through an infinite nozzle of
  arbitrary cross-sections}, Archive for Rational Mechanics and Analysis, 184
  (2007), pp.~185--242.

\bibitem{chen2009uniqueness}
{\sc G.-Q. Chen and H.~Yuan}, {\em Uniqueness of transonic shock solutions in a
  duct for steady potential flow}, Journal of Differential Equations, 247
  (2009), pp.~564--573.

\bibitem{chen2008trans}
{\sc S.~Chen}, {\em Transonic shocks in 3-d compressible flow passing a duct
  with a general section for euler systems}, Transactions of the American
  Mathematical Society, 360 (2008), pp.~5265--5289.

\bibitem{chen2008transonic}
{\sc S.~Chen and H.~Yuan}, {\em Transonic shocks in compressible flow passing a
  duct for three-dimensional euler systems}, Archive for Rational Mechanics and
  Analysis, 187 (2008), pp.~523--556.

\bibitem{courant1999supersonic}
{\sc R.~Courant and K.~O. Friedrichs}, {\em Supersonic flow and shock waves},
  vol.~21, Springer Science \& Business Media, 1999.

\bibitem{evans2002partial}
{\sc L.~C. Evans}, {\em Partial differential equations, ams}, Graduate Studies
  in Mathematics, 19 (2002).

\bibitem{gilbarg2015elliptic}
{\sc D.~Gilbarg and N.~S. Trudinger}, {\em Elliptic partial differential
  equations of second order}, springer, 2015.

\bibitem{han2011elliptic}
{\sc Q.~Han and F.~Lin}, {\em Elliptic partial differential equations}, vol.~1,
  American Mathematical Soc., 2011.

\bibitem{li2010transonic}
{\sc J.~Li, Z.~Xin, and H.~Yin}, {\em On transonic shocks in a conic divergent
  nozzle with axi-symmetric exit pressures}, Journal of Differential Equations,
  248 (2010), pp.~423--469.

\bibitem{liu2016stability}
{\sc L.~Liu, G.~Xu, and H.~Yuan}, {\em Stability of spherically symmetric
  subsonic flows and transonic shocks under multidimensional perturbations},
  Advances in Mathematics, 291 (2016), pp.~696--757.

\bibitem{liu2009global}
{\sc L.~Liu and H.~Yuan}, {\em Global uniqueness of transonic shocks in
  divergent nozzles for steady potential flows}, SIAM Journal on Mathematical
  Analysis, 41 (2009), pp.~1816--1824.

\bibitem{xin2008transonic}
{\sc Z.~Xin and H.~Yin}, {\em The transonic shock in a nozzle, 2-d and 3-d
  complete euler systems}, Journal of Differential Equations, 245 (2008),
  pp.~1014--1085.

\end{thebibliography}
\end{document}